\documentclass[11pt]{article}
\pdfoutput=1
\usepackage{amsbsy,amsmath,amsthm,amssymb,graphicx,verbatim}
\usepackage{setspace, float}
\usepackage[longnamesfirst,sort]{natbib}

\usepackage{multirow}
\usepackage{subcaption}
\usepackage{hhline,xcolor}

\usepackage{geometry}
\newtheorem{theorem}{Theorem}

\newtheorem{lemma}{Lemma}
\newtheorem{proposition}{Proposition}
\newtheorem{assumption}{Assumption}
\theoremstyle{remark}

\newfont{\msbm}{msbm10 at 11pt}

\newcommand{\Cov}[0]{\text{Cov}}
\newcommand{\Var}[0]{\text{Var}}
\newcommand{\E}[0]{\text{E}}
\newcommand\numberthis{\addtocounter{equation}{1}\tag{\theequation}}

\allowdisplaybreaks

\makeatletter
\newcommand{\distas}[1]{\mathbin{\overset{#1}{\kern\z@\sim}}}%
\makeatother

\newcommand{\RN}[1]{%
  \textup{\uppercase\expandafter{\romannumeral#1}}%
}

\begin{document}
\onehalfspacing

\title{Batch size selection for variance estimators in MCMC}

\author{Ying Liu \\ Department of Statistics \\ University of California, Riverside \\ {\tt yliu055@email.ucr.edu}\and Dootika Vats\thanks{Research supported by National Science Foundation} \\ Department of Statistics \\ University of Warwick \\ {\tt dootika.vats@gmail.com} \and James M. Flegal \\ Department of Statistics \\ University of California, Riverside \\ {\tt jflegal@ucr.edu}   }  

\date{\today}
\maketitle
\begin{abstract}
We consider batch size selection for a general class of multivariate batch means variance estimators, which are computationally viable for high-dimensional Markov chain Monte Carlo simulations.  We derive the asymptotic mean squared error for this class of estimators.  Further, we propose a parametric technique for estimating optimal batch sizes and discuss practical issues regarding the estimating process.  Vector auto-regressive, Bayesian logistic regression, and Bayesian dynamic space-time examples illustrate the quality of the estimation procedure where the  proposed optimal batch sizes outperform current batch size selection methods.
\end{abstract}

\section{Introduction} \label{sec:intro}
In Markov chain Monte Carlo (MCMC) simulations, estimating the variability of ergodic averages is critical to assessing the quality of estimation \citep[see e.g.][]{fleg:hara:jone:2008, geye:2011, jone:hobe:2001}. Estimation of this variability can be approached through a multivariate Markov chain central limit theorem (CLT).  To this end, let $F$ be a probability distribution with support $\mathsf{X} \subseteq \mathbb{R}^{d}$ and $g: \mathsf{X}\rightarrow \mathbb{R}^{p}$ be an $F$-integrable function.  Suppose we are interested in estimating the $p$-dimensional vector
\[
\theta := \int_{\mathsf{X}} g(x) \,dF,
\]
using draws from a Harris $F$-ergodic Markov chain, say $\{X_t\}$. For $Y_{t}=g(X_{t}),t \ge 1$, $\bar{Y} = n^{-1}\sum_{t=1}^{n} Y_{t}\rightarrow\theta$ with probability 1 as $n\rightarrow\infty$. Let $\Sigma:=\sum_{k=-\infty}^{\infty}\text{Cov}_{F}(Y_{1}, Y_{1+k})$. The sampling distribution for $\bar{Y} - \theta$ is available via a Markov chain CLT
\begin{equation*}
\sqrt{n}(\bar{Y}-\theta)\xrightarrow{d}N_{p}(0, \Sigma) \text{ as } n\rightarrow\infty\,.
\end{equation*}
We assume throughout this CLT holds \citep[see e.g.][]{jone:2004} and consider estimation of $\Sigma$.  Three popular classes of estimators of $\Sigma$ are spectral variance, (non-overlapping) batch means (BM), and overlapping batch means (OBM).  
Part of our contribution is studying multivariate expressions of generalized OBM estimators of $\Sigma$. 

All three classes of estimators account for serial correlation in the Markov chain up to a certain lag.  This lag, denoted as $b$, is called the bandwidth and batch size in spectral variance and (O)BM estimators, respectively.  The choice of $b$ is crucial to finite sample performance, but choosing $b$ has not been carefully addressed in MCMC. A large batch size yields high variability in the estimator and a small batch size can lead to significant underestimation of $\Sigma$. A batch size of $b = \lfloor n^{1/2} \rfloor$, suggested by \cite{fleg:jone:2010}, is often used in practice or as a default in software, like our \texttt{R} package \texttt{mcmcse} \citep{mcmcse:2017}.  Such a batch size is suboptimal since the mean square error (MSE) optimal batch size for estimators we consider here is proportional to $n^{1/3}$ where the proportionality constant requires estimation \citep{song:schm:1995, dame:1995, fleg:jone:2010}.  We carefully consider batch size selection for MCMC simulations and provide computationally viable improvements over current batch size practices. 

First, we present a multivariate version of the generalized OBM estimator of \cite{dame:1991}. This is a substantial generalization of the traditional OBM estimator since it allows the flexibility of using different lag windows. We obtain an MSE optimal batch size expression for this class of estimators. The resulting bias and variance expressions mirror those of spectral variance estimators \citep{andr:1991}.  However, this estimator computes faster and the conditions presented here are standard in MCMC.

The most common estimators for $\Sigma$ in MCMC are BM estimators, where MSE optimal batch sizes are proportional to $n^{1/3}$. For BM and OBM estimators and when MSE optimal batch sizes exist, we provide a stable and fast estimation procedure for the proportionality constant in the optimal batch size. Our parametric approach caters to MCMC applications with long run lengths.  In short, we use a stationary autoregressive process of order $m$ to approximate the marginals of $\{Y_t\}$, which yields a closed form expression for the unknown proportionality constant. We combine these univariate estimators by modifying the weighting system of \cite{andr:1991}.  We compare finite sample performance of our method to nonparametric pilot estimators \citep{poli:roma:1999, politis:dimitris:2011, politis2003adaptive}.  

Integral to our theoretical and practical results is the choice of lag window used in the weighted BM and generalized OBM estimators. Although linear lag windows are non-optimal, they are particularly useful in long MCMC simulations due to superior computational performance. For this reason, we focus on the Bartlett and flat-top lag windows \citep{poli:roma:1995,poli:roma:1996}. The Bartlett lag window corresponds to traditional BM and OBM estimators, while the flat-top lag window yields alternative BM and OBM estimators intended for bias-correction. For flat-top lag windows the MSE optimal criterion results in a batch size of 0, which is clearly inappropriate. We investigate using Bartlett-optimal batch sizes in this case and compare them with an empirical lag-based method.

Batch size selection has been studied in other contexts such as heteroskedasticity and autocorrelation consistent (HAC) covariance matrices, nonparametric density, and spectral density function estimation.  Broadly speaking, these results are not computationally viable for high-dimensional MCMC where long run lengths are standard.  For example, \cite{andr:1991} obtains MSE optimal bandwidths for spectral variance estimators for HAC estimation.  \cite{politis2003adaptive, politis:dimitris:2011} and \cite{poli:roma:1999} discuss bandwidth selection for spectral variance estimators for the flat-top window function. \cite{chan2017automatic} consider recursive estimation of the time-average variance constant where batch sizes are suggested. An interested reader is directed to \cite{Jone:Chri:Jame:Shea:Simon:1996}, \cite{silv:1999},  \cite{Wood:Mich:1970}, 
and \cite{shea:simo:jone:mich:1991} for bandwidth selection in density estimation.  

We illustrate the quality of our estimation procedures via three examples.  First, a vector autoregressive process of order 1 is examined where the optimal batch size is known. Next, we present a Bayesian logistic regression example and compare the performance of the optimal batch size methods with the more commonly used batch sizes of $\lfloor n^{1/3} \rfloor$ and $\lfloor n^{1/2} \rfloor$. A similar analysis is done for a  Bayesian dynamic space-time model.  

Overall, the simulation studies show a significant improvement in accuracy compared to simply choosing a batch sizes equal to $\lfloor n^{1/3} \rfloor$ or $\lfloor n^{1/2} \rfloor$.  Further, our procedures require limited additional computational effort.  For long run lengths, we recommend BM with an MSE optimal batch size estimated via an autoregressive process of order $m$. For shorter run lengths, flat-top estimators are more robust to the choice of  batch size, as long as the batch size is not unreasonably small.  In the near future we will incorporate these recommendations into the \texttt{mcmcse} \texttt{R} package.


The rest of this paper is organized as follows. Section~\ref{sec:batch_means_and_spectral_variance_estimators} presents generalized OBM estimators and MSE results focusing on Barlett and flat-top lag windows.  Section~\ref{sec:constant} discusses practical batch size selection and proposes a parametric estimation technique for the proportionality constant.  Section~\ref{sec:examples} compares performances between suggested and more commonly used batch sizes in three examples.  We conclude with a discussion in Section~\ref{sec:discussion}.  The proofs establishing bias and variance for generalized OBM variance estimators are relegated to the appendices.

\section{Generalized OBM estimator} 
\label{sec:batch_means_and_spectral_variance_estimators}

We consider the generalized OBM estimator of $\Sigma$ constructed using outer products from means inside batches and a lag window function $w_{n}:\mathbb{Z} \to \mathbb{R}$. The lag window $w_n$ is a function that assigns weights to the lags and is integral to spectral variance estimators. 
Define $\Delta_{1}w_{n}(k)=w_{n}(k-1)-w_{n}(k)$ and $\Delta_{2}w_{n}(k)=w_{n}(k-1)-2w_{n}(k)+w_{n}(k+1)$.  For a Monte Carlo sample size $n$, let $\bar{Y}_{l}(k)=k^{-1}\sum_{t=1}^{k}Y_{l+t}$ for $l=0,..., n-k$ and consider
\begin{equation} \label{eq:gobm}
\hat{\Sigma}_w=\dfrac{1}{n}\sum_{k=1}^{b}\sum_{l=0}^{n-k}k^{2}\Delta_{2}w_{n}(k)  \left(\bar{Y}_{l}(k)-\bar{Y} \right) \left(\bar{Y}_{l}(k)-\bar{Y} \right)^{T},
\end{equation}
with components $\hat{\Sigma}_{w,ij}$.  
\cite{dame:1991} proposed the generalized OBM estimator for $p=1$, which was also studied in \cite{atch:2011} and \cite{fleg:jone:2010}.  \cite{vats:fleg:jone:2015spec} generalized the estimator for $p>1$ and used it to establish strong consistency of multivariate spectral variance estimators in MCMC.  \cite{liu:fleg:2018} propose a nonoverlapping version of \eqref{eq:gobm}, referred to as weighted BM estimators.  

We assume throughout that the lag window is an even function defined on $\mathbb{Z}$ such that (i) $|w_{n}(k)| \leq 1$ for all $n$ and $k$, (ii) $w_{n}(0) = 1$ for all $n$, and (iii) $w_{n}(k) = 0$ for all $|k|\geq b$.  Figure~\ref{fig:windows} illustrates the following three lag windows:
\begin{align*}
\text{Bartlett:} \quad w_{n}(k) & = \left( 1-|k|/b \right) I \left( |k| \le b \right) , \\
\text{(Bartlett) Flat-top:} \quad  w_{n}(k) & = I \left( |k| \le b/2 \right) + \left( 2(1-|k|/b) \right)  I \left( b/2 < |k| \le b \right) \text{, and} \\
\text{Tukey-Hanning:} \quad w_{n}(k) & = \left( (1+\text{cos}(\pi |k|/b))/2 \right)  I \left( |k| \le b \right)\,.
\end{align*}
We restrict our attention to the Bartlett and flat-top lag windows since their linearity implies computational efficiency.
\begin{figure}[h]
    \centering
    \includegraphics[width=0.5\textwidth]{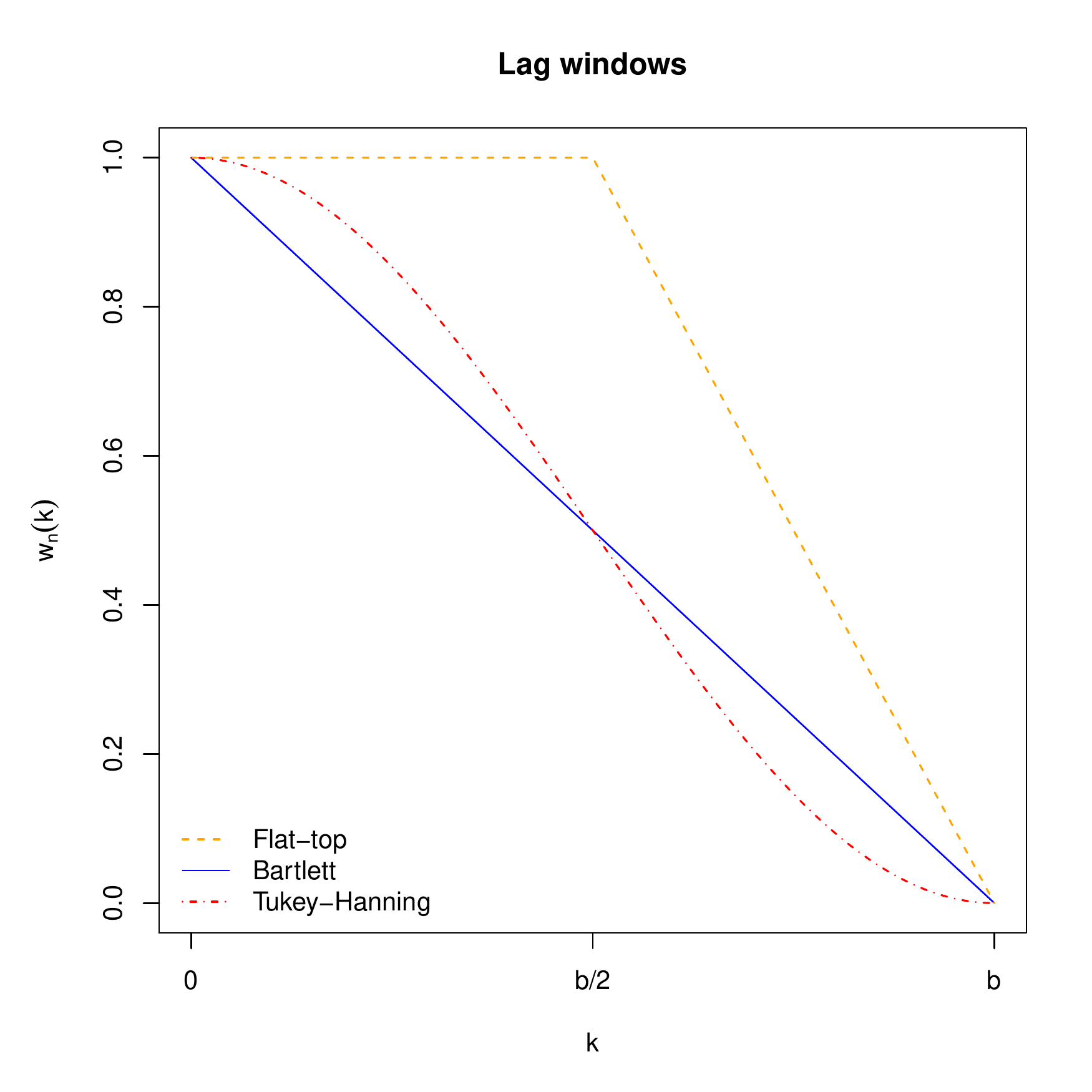}
    \caption{Plot of Bartlett, Tukey-Hanning, and flat-top lag windows.}
    \label{fig:windows}
\end{figure}

Strong and mean square consistency require $b$ and $n/b$ to increase with $n$. We assume the following throughout.
\begin{assumption} \label{ass:batch}
The batch size $b$ is an integer sequence such that $b\rightarrow\infty$ and $n/b\rightarrow\infty$ as $n\rightarrow\infty$, where $b$ and $n/b$ are both monotonically non-decreasing.
\end{assumption}

Lemma~\ref{lem:sip} establishes a strong invariance principle for polynomially ergodic Markov chains. Under the conditions of Lemma~\ref{lem:sip} and conditions on the lag windows, \cite{vats:fleg:jone:2015spec} showed that $\hat{\Sigma}_w$ is strongly consistent for $\Sigma$. Strong consistency is useful for demonstrating asymptotic validity of confidence regions constructed via sequential stopping rules \citep{glyn:whit:1992, vats:fleg:jone:2015output}. Let $\| \cdot\|$ denote the Euclidean norm. 

\begin{lemma} \label{lem:sip}
\citep{vats:fleg:jone:2015spec} Let $f: \mathsf{X} \to \mathbb{R}^p$ be such that $\E_F\|f(X)\|^{2 + \delta} < \infty$ for some $\delta > 0$ and let $\{X_t\}$ be a polynomially ergodic Markov chain of order $\xi \geq (1 + \epsilon)(1 + 2/\delta)$ for some $\epsilon > 0$. Let $B(n)$ be a $p$-dimensional standard Brownian motion and $L$ be a $p \times p$ lower triangular matrix. Then for some $\lambda > 0$ and a finite random variable $D$, with probability 1
	\[
\left\|\sum_{t=1}^{n} f(X) - n \E_F f - L B(n) \right \|  < D n^{1/2 - \lambda}\,.
	\]
\end{lemma}


Using $\hat{\Sigma}_w$ requires selecting a batch size, $b$. Large batch sizes capture more lag correlations yielding larger variance, while small batch sizes yield higher bias.  Theorems~\ref{thm:bias_sv} and~\ref{thm:var} in Appendix~\ref{sec:optimal} derive the element-wise asymptotic bias and variance for $\hat{\Sigma}_w$, respectively, which we summarize below as the element-wise MSE of $\hat{\Sigma}_w$. Denote the components of $\Sigma$ as $\Sigma_{ij}$ and the lag $k$ autocovariance by $R(k) = \E_{F} \left( Y_{t} - \theta \right) \left( Y_{t+k} - \theta \right)^{T}$.  Further, define
\begin{equation*}
\Gamma=-\sum_{k=1}^{\infty}k \left[ R(k)+R(k)^T \right]
\end{equation*}
with components $\Gamma_{ij}$.
\begin{theorem}
	\label{thm:mse_obm}
Let the conditions of Lemma~\ref{lem:sip} hold for $f = g$ and $f = g^2$ (where the square is element-wise) such that $\E_FD^4 < \infty$ and $\E_F \|g \|^{4 + \delta} < \infty$ for some $\delta > 0$. Further suppose
\begin{enumerate}
\item $\sum_{k=1}^{b} k \Delta_2 w_n(k) = 1$,
\item $
\sum_{k=1}^{b}(\Delta_{2}w_{k})^2 =  O\left(1/b^2\right),
$
\item $b \,n^{1 - 2\lambda} \left(\sum_{k=1}^{b}|\Delta_{2}w_{n}(k)| \right)^{2} \log n\rightarrow 0$, \text{ and}
\item $n^{1 - 2\lambda} \sum_{k=1}^{b}|\Delta_{2}w_{n}(k)|\rightarrow 0$\,.
\end{enumerate}
Then, for $C = b \sum_{k=1}^{b}\Delta_{2}w_{n}(k)$ and $S \ne 0$ that depends on the lag window,
\begin{align*}
\text{MSE}\left(\hat{\Sigma}_{w,ij}\right)
&=\dfrac{C^2\Gamma^{2}_{ij}}{b^{2}} + [\Sigma_{ii}\Sigma_{jj}+\Sigma_{ij}^{2}] S\dfrac{b}{n}+o\left(\dfrac{b}{n}\right)+o\left(\dfrac{1}{b}\right)\,.  \numberthis \label{eq:obm_mse}
\end{align*}
\end{theorem}
Assumption~\ref{ass:batch} holds for batch sizes proportional to $\lfloor n^{\nu} \rfloor$ where $0 < \nu < 1$.  Then $\text{MSE}\left(\hat{\Sigma}_{w,ij}\right) \to 0$, as $n \to \infty$, if $1 - 2 \lambda - \nu < 0$.  The constant $\lambda$ is related to the mixing rate of the Markov chain \cite[see][]{kuelbs:phil:1980,dame:1991}; a value closer to 1/2 indicates fast mixing of the process.  

To obtain the MSE optimal batch size for a particular lag window, we minimize the MSE expression in \eqref{eq:obm_mse}. Both $C$ and $S$ depend on the choice of lag window, where $S$ may depend on $b$ for non-linear lag windows. We focus on linear lag windows due to their computational feasibility.  That is, MSE optimal batch sizes for the Bartlett and flat-top windows are
\begin{equation}
\label{eq:optimal batch}
b_{opt,ij}=\left(\dfrac{2C^2}{S}\dfrac{\Gamma_{ij}^2n}{\Sigma_{ii}\Sigma_{jj}+\Sigma_{ij}^{2}}\right)^{1/3}\,.
\end{equation}
For these lag windows, \cite{andr:1991} obtains the same MSE for spectral variance estimators in linear regression settings with heteroscedastic, temporally dependent errors of unknown form.  However, the formulation at \eqref{eq:gobm} is more computationally efficient.

\subsection{Bartlett}

The Bartlett lag window is by far the most common, see e.g.\ \cite{newe:west:1987}.  For this lag window, $\Delta_{2}w_{n}(b)=1/b$ and $\Delta_{2}w_{n}(k)=0$ for all other $k$ values.  Then, $C = 1$ and the double summation in \eqref{eq:gobm} reduces to
\[
\hat{\Sigma}_{B}(b) =\dfrac{b}{n}\sum_{l=0}^{n-b}  \left(\bar{Y}_{l}(b)-\bar{Y}  \right)  \left(\bar{Y}_{l}(b)-\bar{Y}   \right)^{T},
\]
which is asymptotically equivalent to the OBM estimator. The conditions of Theorem~\ref{thm:mse_obm} are satisfied since $\sum_{k=1}^{b} k \Delta_2 w_n(k) = 1$, $\sum_{k=1}^{b} (\Delta_2 w_b(k))^2 = 1/b^2$, and $\sum_{k=1}^{b} |\Delta_2 w_b(k))| = 1/b$. Finally, $S = 2/3$ (see Appendix~\ref{sec:optimal}) so the MSE optimal batch size at \eqref{eq:optimal batch} is
\begin{equation*}
\label{eq:obm_optimal}
b_{opt,ij}=\left(\dfrac{3\Gamma_{ij}^2n}{\Sigma_{ii}\Sigma_{jj}+\Sigma_{ij}^{2}}\right)^{1/3}\,.
\end{equation*}

\cite{liu:fleg:2018} propose a nonoverlapping version of \eqref{eq:gobm}, referred to as weighted BM estimators.  When the Bartlett lag window is used within weighted BM, the result is the commonly used BM estimator of $\Sigma$, which we now describe. Let $n = ab$, where $a$ is the number of batches and $b$ is the batch size. For $l = 0, \dots a-1$, let $\bar{Y}_l = b^{-1} \sum_{t=1}^{b} Y_{lb+t}$ denote the mean vector of the batch. Then the BM estimator is
\begin{equation*}
\dot{\Sigma}_B(b) = \dfrac{b}{a-1} \sum_{l=0}^{a-1} (\bar{Y}_l - \bar{Y})(\bar{Y}_l - \bar{Y})^T\,.
\end{equation*}
The MSE for the BM estimator can be obtained by setting $r = 1$ in Theorems 5 and 6 of \cite{vats:fleg:2018}.  Specifically, the form is identical to \eqref{eq:optimal batch} with $C = 1$ and $S = 1$.  

\subsection{Flat-top}
Flat-top lag windows do not downweight small lag terms by setting $w_{n}(k)=1$ for $k$ near $0$ \citep[see e.g.][]{poli:roma:1995,poli:roma:1996}.  It is easy to show $\Delta_{2}w_{n}(b/2)=-2/b,\ \Delta_{2}w_{n}(b)=2/b$, and $\Delta_{2}w_{n}(k)=0$ for all other $k$ values.  Hence, $C = 0$ and for even $b$ in \eqref{eq:gobm} we have
\[
\hat{\Sigma}_{F}(b) = 2\,\hat{\Sigma}_{B}(b) - \hat{\Sigma}_{B}(b/2).
\]
That is, using the flat-top lag window in \eqref{eq:gobm} gives a linear combination of OBM estimators.  In addition, $\sum_{k=1}^{b} k \Delta_2 w_n(k) = 1$, $\sum_{k=1}^{b} (\Delta_2 w_b(k))^2 = 8/b^2$, and $\sum_{k=1}^{b} |\Delta_2 w_b(k))| = 4/b$, so the conditions of Theorem~\ref{thm:mse_obm} are satisfied. Moreover, $S=4/3$ (see Appendix~\ref{sec:optimal}) so the MSE in \eqref{eq:optimal batch} is a strictly decreasing function of $b$ yielding an MSE optimal batch size of 0.


The flat-top lag window within weighted BM yields a linear combination of BM estimators, i.e., for even $b$
\[
\dot{\Sigma}_F(b) = 2\,\dot{\Sigma}_B(b) - \dot{\Sigma}_B(b/2)\,.
\]
Setting $r = 2$ in Theorems 5 and 6 of \cite{vats:fleg:2018}, the MSE is identical to \eqref{eq:obm_mse} with $C = 0$ and $S=5/2$ in \eqref{eq:optimal batch}.  Again the MSE is a decreasing function of $b$ implying an unrealistic optimal batch size of 0. 

Flat-top lag windows provide bias-corrected variance estimators at the cost of slightly higher variance, thus a MSE optimal criterion for batch size selection is undesirable for these lag-windows.  Instead we consider a lag-based method of choosing $b$ from \cite{poli:roma:1995}, which is presented in the next section.

\section{Batch sizes in practice} \label{sec:constant}

Informed batch size selection requires knowledge of the underlying process.  For optimal batch sizes at \eqref{eq:optimal batch}, this knowledge is contained in $\Gamma$ and $\Sigma$.  A common solution is to estimate these via a pilot run \citep[see e.g.][]{Wood:Mich:1970, Jone:Chri:Jame:Shea:Simon:1996, load:cliv:1999}. Two such procedures are the nonparametric empirical rule \citep{politis2003adaptive} and the iterative plug-in estimator \citep{broc:mich:gass:theo:herr:1993,buhlmann1996locally}. In both, a spectral variance estimator is constructed where the bandwidth is chosen by an empirical or iterative rule that monitors lag autocorrelations.  We do not require consistency for estimators of $b_{opt,ij}$, thus the pilot step need not be based on BM or spectral variance estimators. 

Estimators of $b_{opt,ij}$ should be computationally inexpensive and have low variability.  The empirical rule and iterative plug-in estimators can be computationally involved (especially for slow mixing chains) and hence they fail the first criteria.  These estimators also exhibit high variability, which we illustrate in our examples.  Low variability is particularly important since the user cannot be expected to run multiple pilot runs.

\subsection{MSE optimal batch sizes}

We provide a parametric estimation technique for estimating $\Gamma$ and $\Sigma$ specifically tailored for MCMC simulations.  Choosing a different $b$ for each element of $\Sigma$ requires substantial computational effort and it is unclear if it makes intuitive sense. Since $b$ can be calculated for each univariate component, we define the overall optimal $b$ by a harmonic-like average of the diagonals $b_{opt, ii}$. That is, we define
\begin{equation*}
b_{opt} \propto \left( \dfrac{\sum_{i=1}^{p}\Gamma_{ii}^2}{ \sum_{i=1}^{p}\Sigma_{ii}^2}\right)^{1/3}n^{1/3} \,,
\end{equation*}
where the proportionality constant is known and depends on the choice of variance estimator and the resulting $C$ and $S$.  This approach is similar to \cite{andr:1991} with a unit weight matrix appropriate for MCMC simulations.  

We now present pilot estimators that are computationally inexpensive and demonstrate low variability. We use a stationary autoregressive process of order $m$ (AR$(m)$) approximation to the marginals of $\{Y_t\}$. For $t = 1, 2, \dots,$ let $W_t \in \mathbb{R}$ be such that
\[
W_t = \sum_{i=1}^{m} \phi_i W_{t-i} + \epsilon_t\,,
\]
where $\epsilon_t$ has mean 0 and variance $\sigma^2_e$, and $\phi_1, \dots, \phi_m$ are the autoregressive coefficients. Let $\gamma(k)$ be the lag $k$ autocovariance function for the process. By the Yule-Walker equations, it is known that for $k > 0$, $ \gamma(k) = \sum_{i=1}^{m} \phi_i \gamma(k-i)$, and $\gamma(0) = \sum_{i=1}^{m} \phi_i \gamma(-i) + \sigma^2_e$. We obtain expressions for $\Sigma_{ii} = \sum_{k=-\infty}^{\infty} \gamma(k)$ and $\Gamma_{ii} = - 2 \sum_{k=1}^{\infty}k \gamma(k)$ and use these to obtain pilot estimates denoted $\Sigma_{p,i}$ and $\Gamma_{p,i}$, respectively. First, it is known that
$$\sum_{k=-\infty}^{\infty} \gamma(k) = \dfrac{\sigma^2_e}{(1 - \sum_{i=1}^{m} \phi_i)^2} \,.$$
%
Following \cite{taylor:2018},
\begin{align*}
\sum_{k=1}^{\infty} k \gamma(k) &= \sum_{k=1}^{\infty} k  \sum_{i=1}^m \phi_i \gamma(k-i) \\
&= \sum_{i=1}^m \phi_i \left( \sum_{k=1}^{\infty} k   \gamma(k-i) \right) \\
&= \left(\sum_{i=1}^m \phi_i  \sum_{k=1}^{i} k   \gamma(k-i) \right) + \left( \sum_{i=1}^m \phi_i \sum_{s=1}^{\infty} (s+i)   \gamma(s) \right) \\
&= \left(\sum_{i=1}^m \phi_i  \sum_{k=1}^{i} k   \gamma(k-i) \right) + \left( \sum_{i=1}^m \phi_i  \sum_{s=1}^{\infty} s   \gamma(s) \right)+ \left( \sum_{i=1}^m \phi_i i \sum_{s=1}^{\infty}    \gamma(s)  \right) \\
\Rightarrow \sum_{k=1}^{\infty} k \gamma(k) & = \left[\left(\sum_{i=1}^m \phi_i  \sum_{k=1}^{i} k   \gamma(k-i) \right) +  \dfrac{(\sigma^2_e - \gamma(0))}{2}\left( \sum_{i=1}^m i \phi_i  \right)  \right] \left( \dfrac{1}{1 - \sum_{i=1}^{m} \phi_i} \right)\,.
\end{align*}

We fit an AR$(m)$ model for each marginal of the Markov chain, where $m$ is determined by Akaike information criterion. The autocovariances $\gamma(k)$ are estimated by the sample lag autocovariances, $\hat{\gamma}(k)$.  Then $\sigma^2_e$ and  $\phi_i$ are estimated by $\hat{\sigma}^2_{e}$ and $\hat{\phi}$, respectively, by solving the Yule-Walker equations. For the $i$th component of the Markov chain, the resulting AR$(m)$-fit estimators are 
\begin{equation*}
	\label{eq:ar_sigma}
\Sigma_{p,i} = \dfrac{\hat{\sigma}^2_e}{(1 - \sum_{i=1}^{m} \hat{\phi}_i)^2}, \text{ and }
\end{equation*}
\begin{equation*}
\Gamma_{p,i} = -2 \left[\left(\sum_{i=1}^m \hat{\phi}_i  \sum_{k=1}^{i} k   \hat{\gamma}(k-i) \right) +  \dfrac{(\hat{\sigma}^2_e - \hat{\gamma}(0))}{2}\left( \sum_{i=1}^m i \hat{\phi}_i  \right) \left( \dfrac{1}{1 - \sum_{i=1}^{m} \hat{\phi}_i} \right) \right]\,.
\end{equation*}

An AR$(m)$-fit is a natural choice over the more common AR$(1)$-fit since the components of $Y_t$ are usually not Markov chains. An AR$(m)$-fit for MCMC has also been studied by \cite{thom:2010} who considers estimating the integrated autocorrelation time of a process. Further, the \texttt{R} package \texttt{coda} \citep{plum:best:cow:2006} uses $\Sigma_{p,i}$ to estimate $\Sigma_{ii}$ when calculating effective sample sizes. 

\subsection{Lag-based methods}

\cite{poli:roma:1995} suggest using a bandwidth equal to $2r$ with $r$ chosen such that the estimated lag correlation at $r$ is less than an upper bound.  We consider the following upper bound of \cite{politis2003adaptive}.  Let $\hat{\rho}_{i}(k)$ be the sample lag $k$ correlation for the $i$th component, and let $\rho(k) = \max_{i}|\hat{\rho}_i|$ be the maximum $k$-lag correlation.  
Then $r$ is the smallest integer for which $|\hat{\rho}(r + s)| < 2\sqrt{\log(n)/n}$, for all $s = 1, 2, \dots, 5$. The resulting bandwidth is essentially the lag beyond which this is no significant correlation.  This cutoff works well in simulations, where we also compare it with the AR$(m)$-fit.
\section{Examples} \label{sec:examples}

\subsection{Vector auto-regressive example}
Consider the $p$-dimensional vector auto-regressive process of order 1 (VAR(1)) 
\[
X_t=\Phi X_{t-1}+\epsilon_t,
\]
for $t=1,2,\dots$ where $X_t\in \mathbb{R}^p$, $\epsilon_t$ are i.i.d.\ $N_p(0,I_p)$ and $\Phi$ is a $p\times p$ matrix.  The Markov chain is geometrically ergodic when the largest eigenvalue of $\Phi$ in absolute value is less than 1 \citep{tjostheim1990non}.  In addition, if $\otimes$ denotes the Kronecker product, the invariant distribution is $N_{p}(0, V)$, where $vec(V)=(I_{p^2}-\Phi\otimes\Phi)^{-1}vec(I_p)$.  Consider estimating $\theta=E{X_1} =0$ with $\bar{Y}=\bar{X}_n$. It is known that $\Sigma = (I_p-\Phi)^{-1}V+V(I_p-\Phi)^{-1}-V$.  It can be shown that $\Gamma = - \left[(I_p -\Phi)^{-2}\Phi V  + V \Phi^T(I_p -\Phi^T)^{-2}\right]$.
Thus, the true optimal batch size coefficient can be obtained using the diagonals of $\Sigma$ and $\Gamma$. 

To ensure geometric ergodicity, we generate the process as follows. Consider a $p\times p$ matrix $A$ with each entry generated from a standard normal distribution, let $B=AA^{T}$ be a symmetric matrix with the largest eigenvalue $\gamma$, then set $\Phi_0=B/(\gamma+0.001)$. We evaluate a series of $\Phi=\rho \Phi_0$, with $\rho=\{.80, .82, .84, \dots,.90\}$, where larger $\rho$ values imply stronger auto-covariance and cross-covariance in the process. We set $p = 3$ and over 1000 replications for each $\rho$, a pilot run of length 1e4 estimates the batch size using various methods, and the final estimation of $\Sigma$ is done using a run length of 1e5.

First, we compare the quality of estimation of $(\sum \Gamma_{ii}^2/ \sum \Sigma_{ii}^2 )$ using the AR$(m)$-fit and the nonparametric pilot estimator for each $\rho$.  Optimal coefficients are computed and MSEs over 1000 replications are plotted in Figure~\ref{fig:var1_batch} with $95\%$ confidence intervals. Estimation quality using the AR$(m)$-fit remains fairly constant as a function of $\rho$, while the nonparametric method yields higher MSE as $\rho$ increases.  
\begin{figure}[]
        \centering
        \includegraphics[width=.45\textwidth]{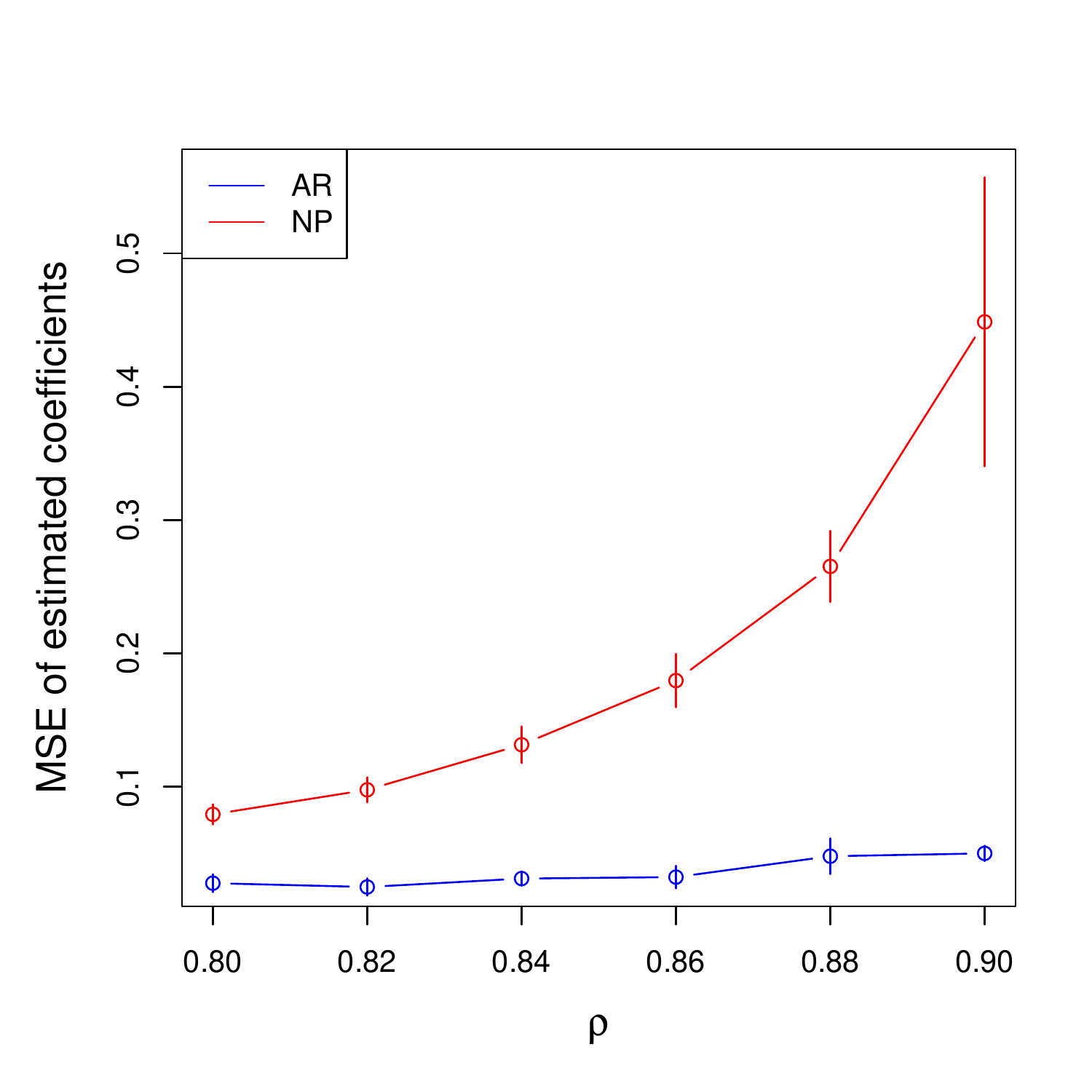}
        \includegraphics[width=.45\textwidth]{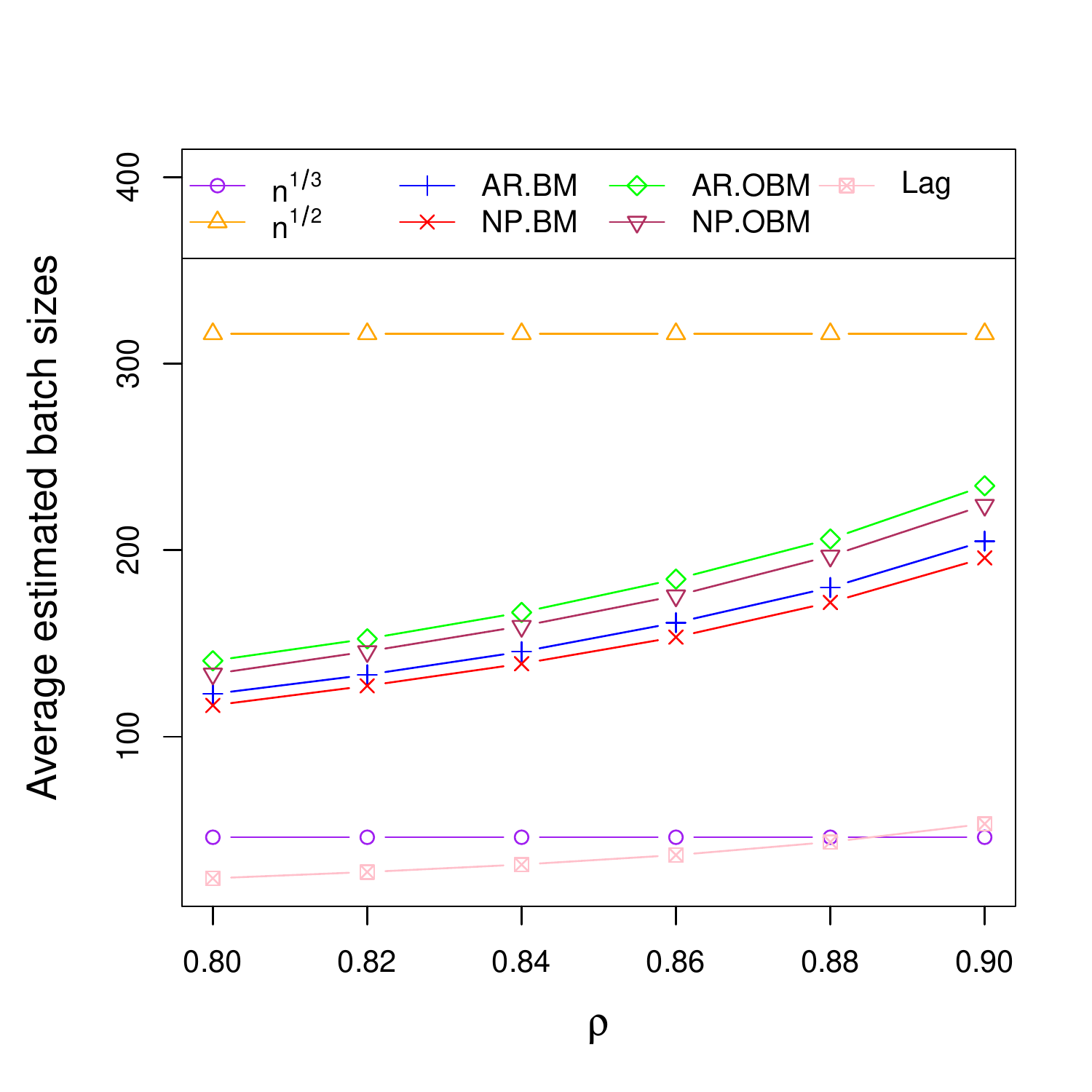}
\caption{Plots of average MSE of the estimated coefficients using the AR$(m)$-fit versus the nonparametric method with $95\%$ CI (left); and estimated batch sizes over $\rho$ using the seven methods (right). }
\label{fig:var1_batch}        
\end{figure}
Also in Figure~\ref{fig:var1_batch} are estimated batch sizes (averaged over 1000 replications) for BM using AR$(m)$-fit (AR.BM), BM using nonparametric method (NP.BM), OBM using AR$(m)$-fit (AR.OBM), OBM using nonparametric method (NP.OBM), and lag-based method. We also plot batch sizes used in practice, $\lfloor n^{1/3} \rfloor$ and $\lfloor n^{1/2} \rfloor$, for comparison. Most noticeably, the lag-based method produces considerably lower batch size estimates for all $\rho$'s, and for all other methods, the average estimated batch sizes are between $\lfloor n^{1/3} \rfloor$ and $\lfloor n^{1/2} \rfloor$. We would then expect $\lfloor n^{1/3} \rfloor$ to yield high bias and for $\lfloor n^{1/2} \rfloor$ to produce estimators with high variability. 

Figure~\ref{fig:var_mse} plots the average MSE across entries of the matrix estimators (in log scale), illustrating how BM and OBM optimal batch sizes lead to smaller MSE than batch sizes $\lfloor n^{1/3} \rfloor$, $\lfloor n^{1/2} \rfloor$, and the lag-based method for the Bartlett window estimators. Our theory discussed in the previous sections agrees with these results. However, for the flat-top window estimators, BM-FT and OBM-FT, $\lfloor n^{1/3} \rfloor$ and the lag-based batch size selection method produces the smallest MSE. It is also apparent from Figure~\ref{fig:var1_batch} these two methods produce smaller batch sizes, which in turn yield smaller MSE for flat-top estimators. This also agrees with our theoretical discussion that MSE decreases as batch sizes decreases for flat-top based estimators. The AR$(m)$-fit and nonparametric methods perform similarly in terms of MSE.

\begin{figure}[htb]
        \centering
        \includegraphics[width=.90\textwidth]{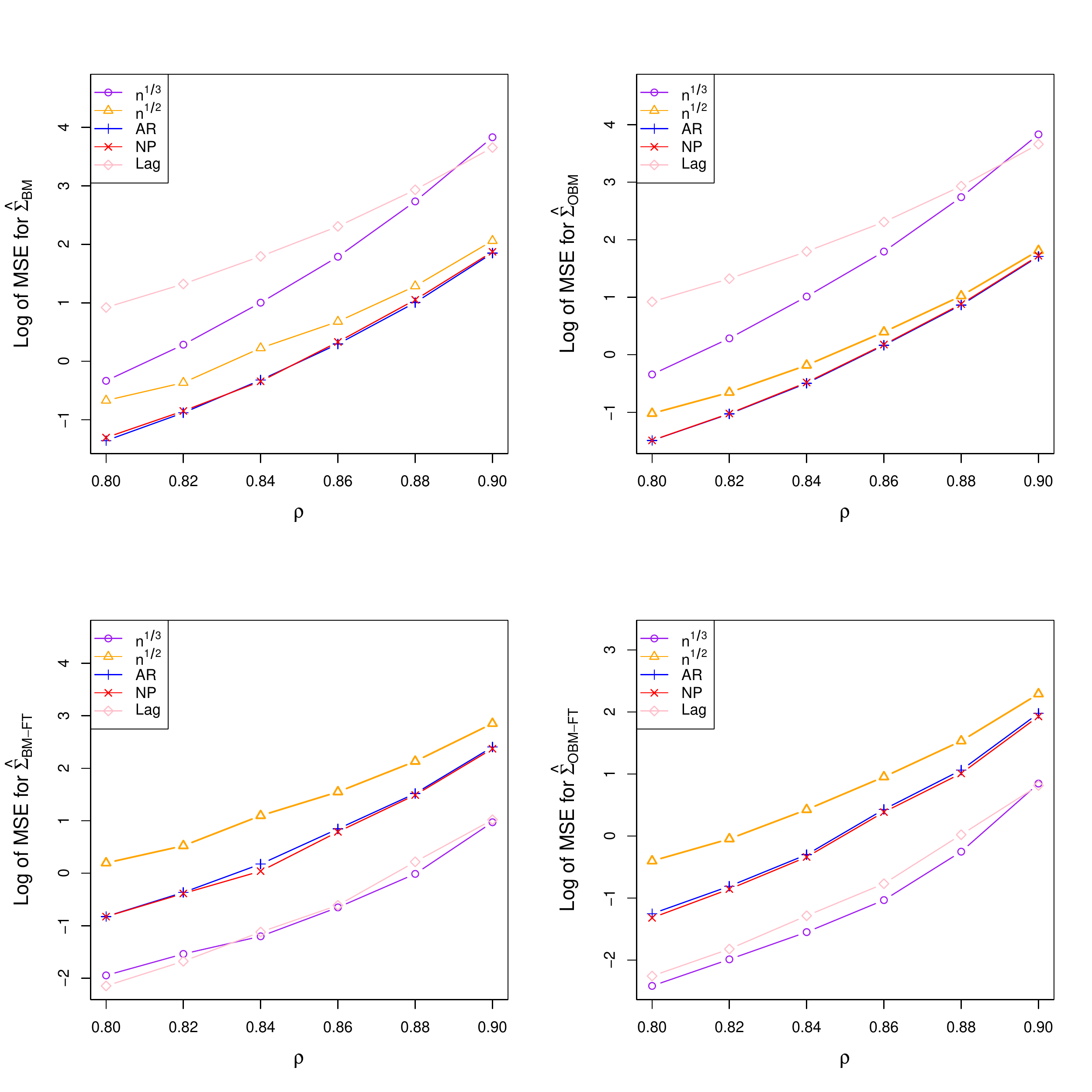}
\caption{Plots of logarithmic MSE for all four estimators using the five batch sizes.}
\label{fig:var_mse}        
\end{figure}


\subsection{Bayesian logistic regression} 
\label{sub:bayes_logistic}
Consider the \textit{Anguilla australis}  data from \cite{elit:leat:2008} available in the \texttt{dismo} R package. The dataset records the presence or absence of the short-finned eel in 1000 sites over New Zealand. Following \cite{leat:elit:2008}, we choose six of the twelve covariates recorded in the data;  SegSumT, DSDist, USNative, DSMaxSlope and DSSlope are continuous and  Method is categorical with five levels.

For $i = 1, \dots, 1000$, let $Y_i$ record the presence ($Y_i = 1$) or absence of \textit{Anguilla australis}. Let $x_i$ denote the vector of covariates for observation $i$. We fit a model with intercept so that the regression coefficient $\beta \in \mathbb{R}^9$. Let 
\[Y_i \mid x_i, \beta \sim \text{Bernoulli} \left( \dfrac{1}{1 + \exp(x_i^T\beta)} \right)  \text{ and } \beta \sim N(0, \sigma^2_{\beta}I_9)\,.\]
We set $\sigma^2_{\beta} = 100$ as in \cite{boon:merr:krac:2014}. The posterior distribution is intractable and we use the \texttt{MCMClogit} function in the R package \texttt{MCMCpack} to obtain posterior samples; this random walk Metropolis-Hastings sampler is geometrically ergodic \citep{vats:fleg:jone:2015output}.

\begin{figure}[]
	\centering
	\includegraphics[width=0.45\textwidth]{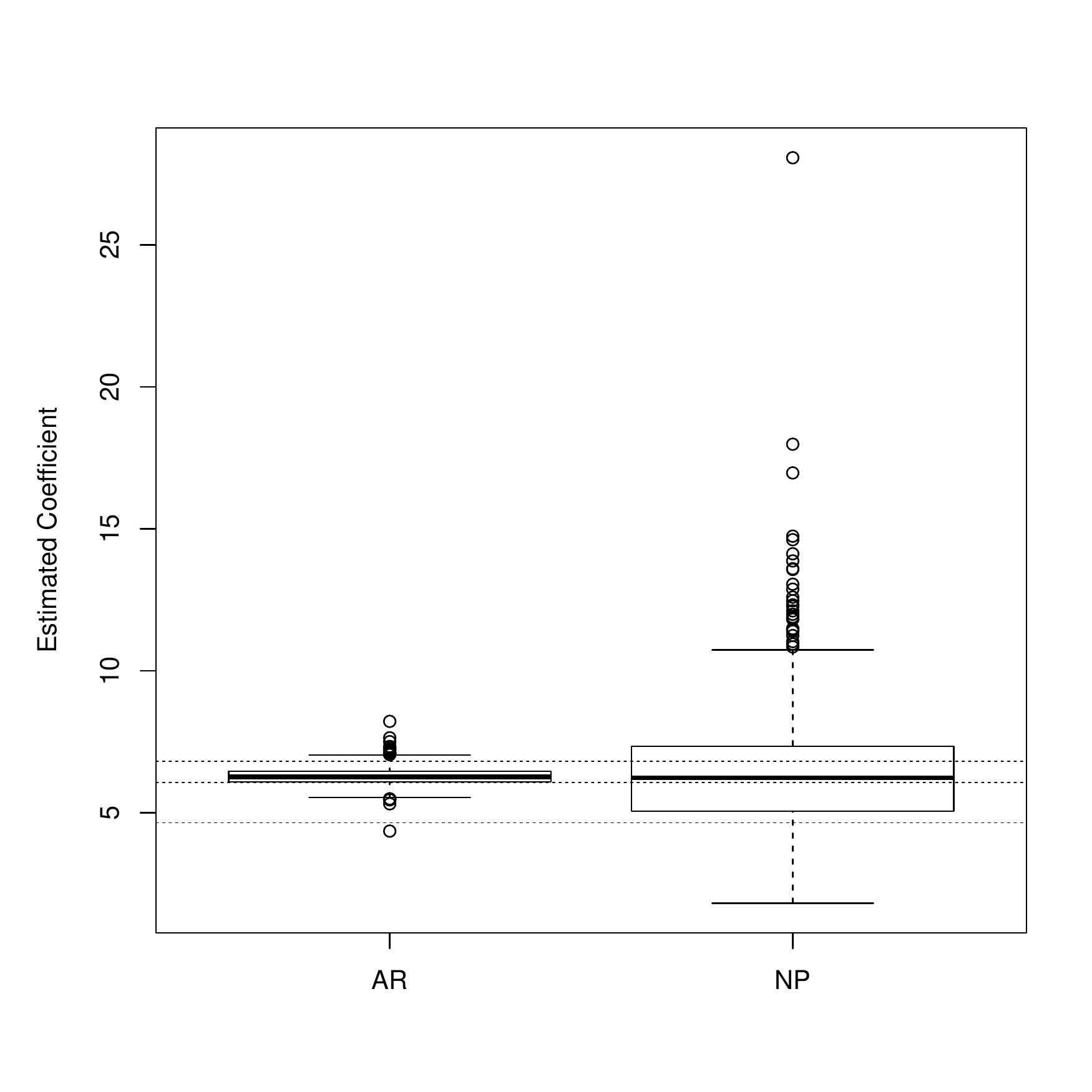}
	\includegraphics[width=0.45\textwidth]{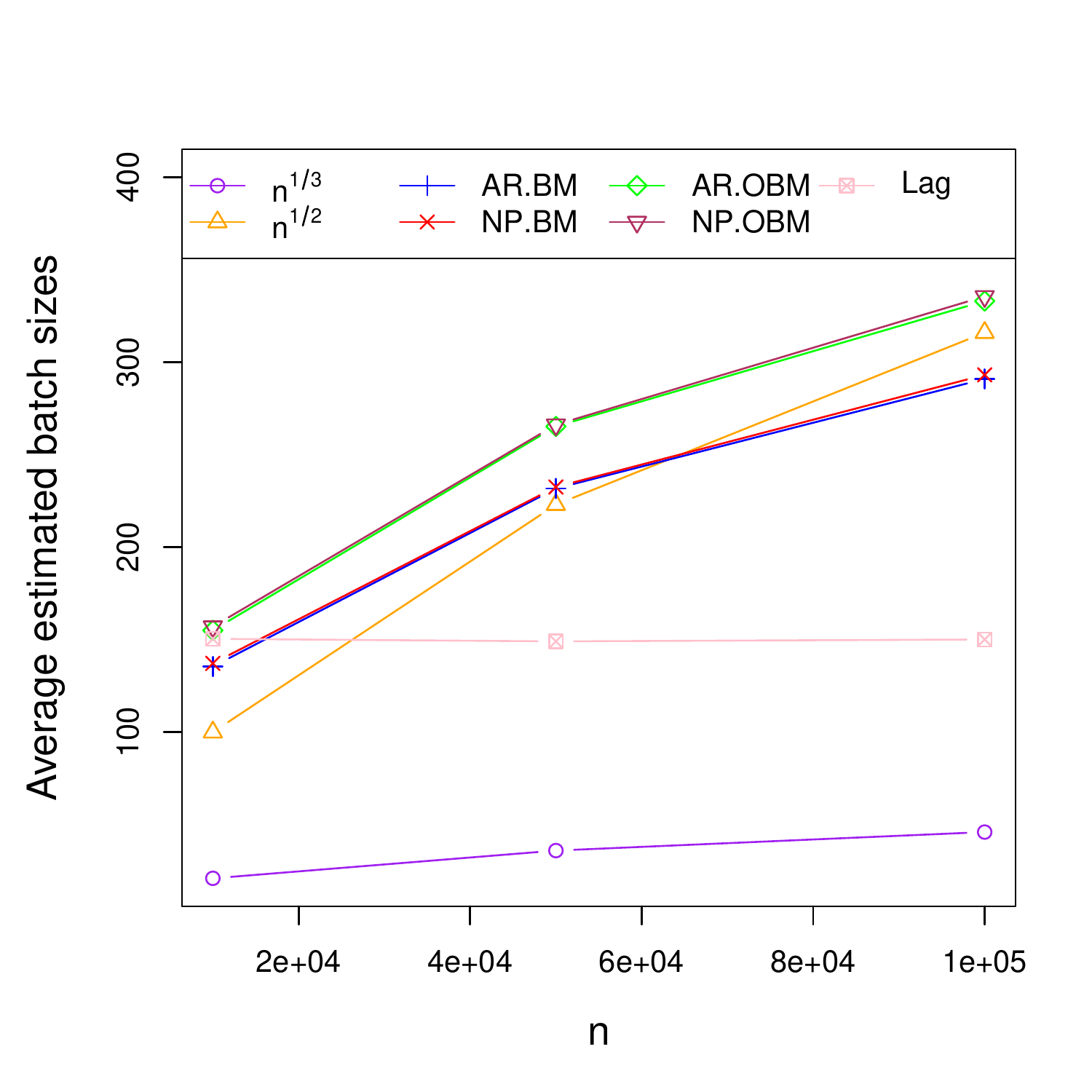}	
	\caption{Boxplot of estimated coefficient of optimal batch size using the AR$(m)$-fit and the nonparametric pilot estimates. The three horizontal lines correspond to $(1e4)^{1/6}, (5e4)^{1/6}$, and $(1e5)^{1/6}$ (left); and estimated batch sizes using the seven methods (right).}
	\label{fig:eel_boxplot_coeff}
\end{figure}

In 1000 replications, we ran a pilot run of length 1e4 to estimate the optimal batch size. We then reran the chain to estimate $\Sigma$ using the estimated optimal batch size for three Monte Carlo sample sizes, $n = 1e4, 5e4,$ and $1e5$. Figure~\ref{fig:eel_boxplot_coeff} presents the variability in the estimates of the coefficient of the optimal batch size and also presents the estimated batch sizes. The AR$(m)$-fit has significantly lower variability compared to the nonparametric pilot estimator. This is particularly useful since a pilot estimator is usually only run once by a user. In addition, since $n^{1/6}$ is close to the estimated coefficients of the optimal batch size, we expect a batch size of $\lfloor n^{1/2} \rfloor$ to perform well for $n = 5e4$ and $n = 1e5$. The plot on the right in Figure~\ref{fig:eel_boxplot_coeff} indicates that for most choices of $n$, $\lfloor n^{1/3} \rfloor$ and the lag-based method yield batch sizes that are relatively small, while the other methods yield high batch sizes. A batch size of $\lfloor n^{1/2} \rfloor$ is fairly close to the estimated batch sizes for these choices of $n$.

Coverage probabilities over the 1000 replications are provided in Table~\ref{tb:coverage_eel}, where the truth is taken to be the average of 1000 MCMC runs of 1e6. Given that the estimated coefficient is significantly larger than 1, it is not surprising that $\lfloor n^{1/3} \rfloor$ performs poorly. For small sample sizes, both the optimal methods have better coverage probabilities. For Monte Carlo sample size $1e5$, as expected, $\lfloor n^{1/2} \rfloor$ fares fairly well. For BM and OBM, the lag-based batch size does not compare well to the optimal methods, but for flat-top based estimators, the lag-based methods perform better than all other methods. The AR$(m)$-fit performs similar to the nonparametric methods based on the coverage probabilities. 

\begin{table}[]
\begin{center}
\begin{tabular}{c|ccccc|ccccc}
\hline
$n$ & $\lfloor n^{1/3} \rfloor $ & $\lfloor n^{1/2} \rfloor $ & AR & NP  & Lag &  $\lfloor n^{1/3} \rfloor $ & $\lfloor n^{1/2} \rfloor $ & AR & NP & Lag \\  \hline
\\

 & \multicolumn{5}{c}{Batch Means} & \multicolumn{5}{c}{Overlapping Batch Means} \\ \hline
1e4 & 0.279  & 0.722  & 0.731  & 0.709  & 0.703  & 0.276  & 0.723  & 0.727  & 0.720 &  0.721 \\
5e4  & 0.499  & 0.826  & 0.823 &  0.831  & 0.808 &  0.494  & 0.832  & 0.837  & 0.831  & 0.813 \\
1e5  & 0.615 &  0.861  & 0.860  & 0.849  & 0.823  & 0.615  & 0.862  & 0.863  & 0.859  & 0.826 \\ \hline
\\
 & \multicolumn{5}{c}{Batch Means - FT} & \multicolumn{5}{c}{Overlapping Batch Means - FT} \\ \hline

1e4  & 0.557  & 0.738  & 0.638  & 0.595  & 0.577  & 0.552  & 0.780  & 0.708  & 0.690  & 0.704 \\
5e4  & 0.753  & 0.814  & 0.829  & 0.820  & 0.849  & 0.760  & 0.854  & 0.851  & 0.842  & 0.876 \\ 
1e5  & 0.825  & 0.857  & 0.854  & 0.854  & 0.882  & 0.827  & 0.877  & 0.877  & 0.880  & 0.887 \\ \hline
\end{tabular}
\caption{Coverage probabilities for $90\%$ confidence regions over 1000 replications for Bayesian logistic regression example.}
\label{tb:coverage_eel}
\end{center}
\end{table}





\subsection{Bayesian dynamic space-time model}
This example considers the Bayesian dynamic model of \cite{finl:bane:gelf:2012} to model monthly temperature data collected at 10 nearby station in northeastern United States in 2000.  A data description can be found in the {\tt spBayes} {\tt R} package \citep{finl:bana:2013}.

Suppose $y_t(s)$ denotes the temperature observed at location $s$ and time $t$ for $s=1,2,..., N_s$ and $t=1,2,..., N_t$. Let $x_t(s)$ be a $m\times 1$ vector of predictors and $\beta_t$ be a $m\times 1$ coefficient vector, which is a purely time component and $u_{t}(s)$ be a space-time component. The model is 
\[ y_t(s) = {\pmb x}_t (s)^{T} {\pmb \beta}_t + u_t (s) + \epsilon_t (s),\ \  \epsilon_t(s) \sim N(0, \tau_t^2),  \]
\[ {\pmb \beta}_t = {\pmb \beta}_{t-1} + {\pmb \eta}_t; \ \ {\pmb \eta}_t \sim N_p (0, \Sigma_{\eta}), \]
\[ u_t(s) = u_{t-1}(s) + w_t(s); \ \ w_t(s) \sim GP (0, C_t(\cdot, \sigma_{t}^{2}, \phi_{t})), \]
where $GP (0, C_t(\cdot, \sigma_{t}^{2},\phi_t))$ is a spatial Gaussian process with $C_t(s_1, s_2; \sigma_t^2,\phi_t) = \sigma_t^2 \rho (s_1, s_2; \phi_t)$, $\rho (\cdot; \phi)$ is an exponential correlation function with $\phi$ controlling the correlation decay, and $\sigma^2_t$ represents the spatial variance components. The Gaussian spatial process allows closer locations to have higher correlations.  Time effect for both $\pmb \beta_t$ and $u_t(s)$ are characterized by transition equations to achieve reasonable dependence structure. The priors on $\theta=(\pmb \beta_t,\ u_t(s),\ \sigma^2_t,\ \Sigma_\eta,\ \tau^2_t,\ \phi_t)$ are the defaults in the {\tt spDynlM} function in the {\tt spBayes} package, and a Metropolis-within-Gibbs sampler is used to sample from the posterior.

The only predictor in the analysis is elevation, hence ${\pmb \beta}_t=(\beta_{t}^{(0)},\beta_{t}^{(1)})^T$ for $t=1,2,...,12,$ where $\beta_{t}^{(0)}$ is the intercept and $\beta_{t}^{(1)}$ is the coefficient for elevation. Consider estimating the coefficient of the covariate for the first two months, $\beta^{(1)}_{1}$ and $\beta^{(1)}_{2}$.

	\begin{figure}[htbp]
	\centering
	\includegraphics[width=0.45\textwidth]{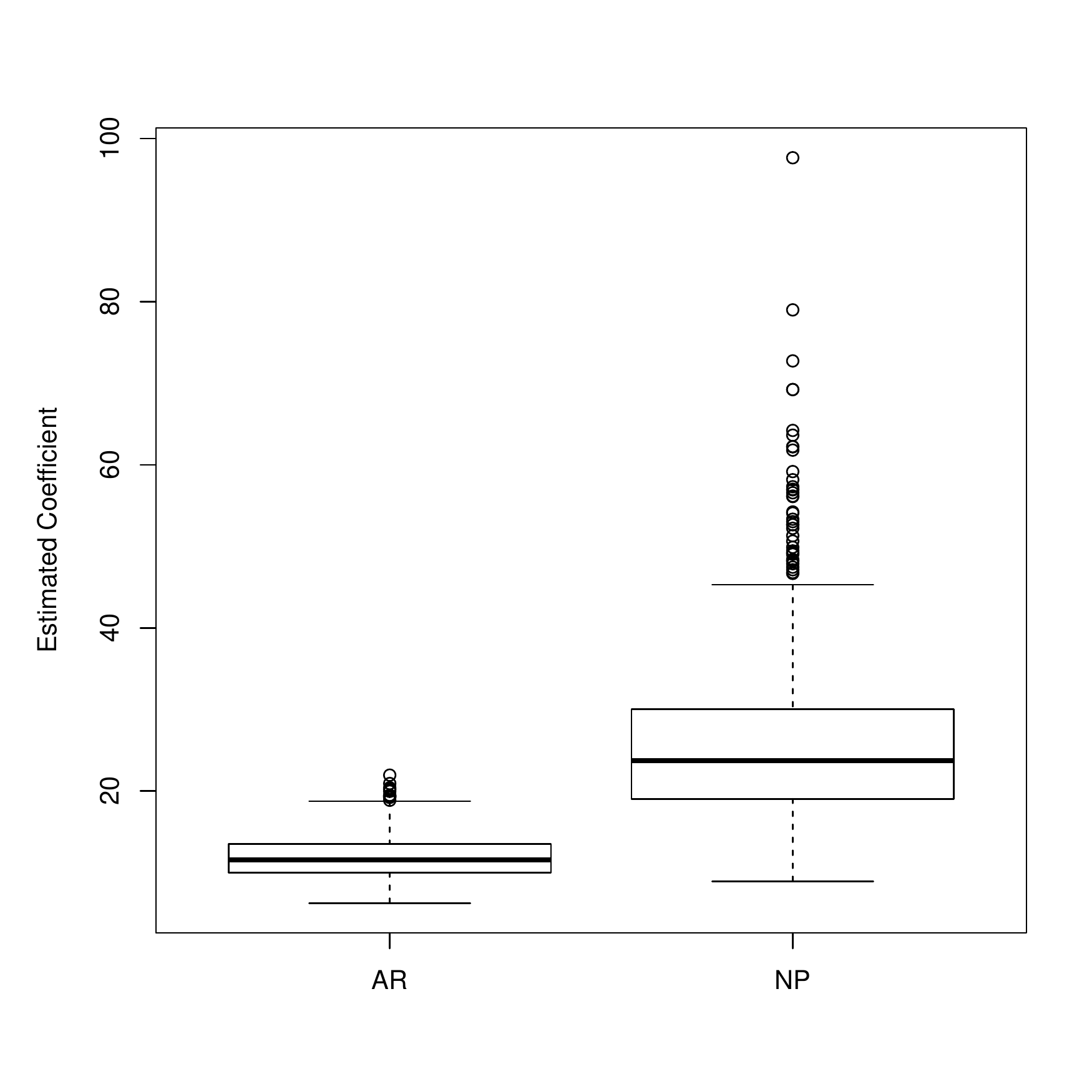}
	   \includegraphics[width=0.45\textwidth]{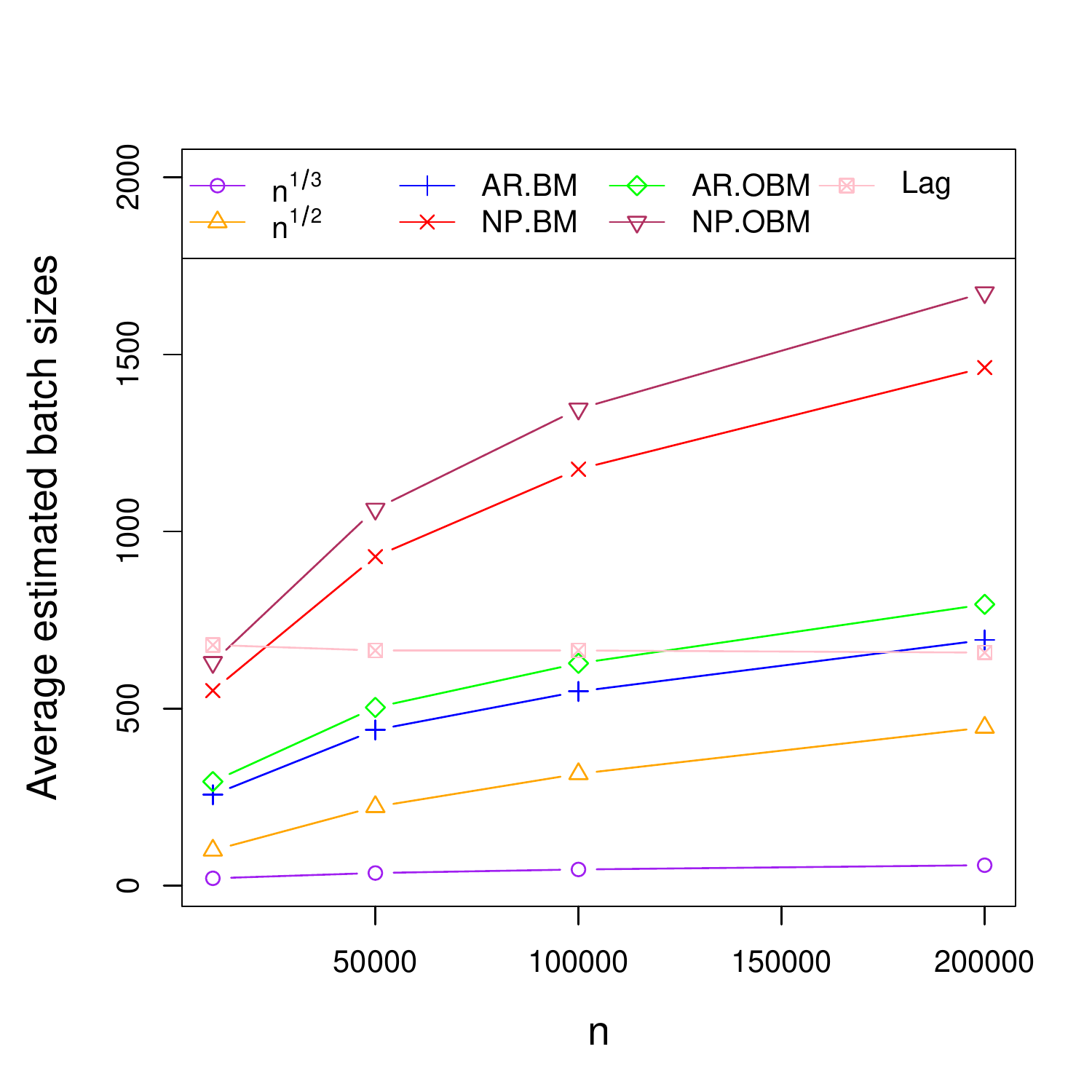}
	\caption{Boxplot of estimated coefficient of optimal batch size using AR$(m)$-fit and nonparametric methods (left); and estimated batch sizes using the seven methods (right).}
	\label{fig:spatial_batch}
\end{figure}

We obtain the true posterior mean of these two components by averaging over 1000 chains of length 1e6. The simulation setup is similar to that in Section~\ref{sub:bayes_logistic}. In Figure~\ref{fig:spatial_batch} we present boxplots of the estimated coefficient of the optimal batch size; here again the variability in the nonparametric estimator is significantly higher, and the central tendency is also significantly higher. In Figure~\ref{fig:spatial_batch} we also present the estimated batch sizes over $n$ for all the methods. The nonparametric method yields batch sizes that are larger than the AR$(m)$-fit. The lag-based method is almost always lower than the OBM optimal batch sizes, and $\lfloor n^{1/3} \rfloor$ and $\lfloor n^{1/2} \rfloor$ are both significantly smaller.

Coverage probabilities over 1000 replications are shown in Table~\ref{tb:coverage_spatial}. Unsurprisingly, the $\lfloor n^{1/3} \rfloor$ and $\lfloor n^{1/2} \rfloor$ do not perform well almost throughout. The nonparametric pilot estimators yield marginally better coverage probabilities and the lag-based methods here do not yield a similar performance as before. We suspect this is due to a shorter pilot run, which may be insufficient to estimate the correlations appropriately. Even so, all methods perform better than currently used batch size choices.

\begin{table}[]
\begin{center}
\begin{tabular}{c|ccccc|ccccc}
\hline
$n$ & $\lfloor n^{1/3} \rfloor $ & $\lfloor n^{1/2} \rfloor $ & AR & NP  & Lag &  $\lfloor n^{1/3} \rfloor $ & $\lfloor n^{1/2} \rfloor $ & AR & NP & Lag \\  \hline
\\

 & \multicolumn{5}{c}{Batch Means} & \multicolumn{5}{c}{Overlapping Batch Means} \\ \hline

1e4 & 0.389  & 0.612  & 0.736  & 0.775  & 0.775  & 0.388  & 0.611  & 0.724  & 0.752  & 0.742 \\
5e4 & 0.439  & 0.732  & 0.804  & 0.842  & 0.816  & 0.440  & 0.728  & 0.806  & 0.841  & 0.810 \\ 
1e5 & 0.477  & 0.764  & 0.820  & 0.841  & 0.810  & 0.477  & 0.761  & 0.819  & 0.839  & 0.811 \\
2e5 & 0.553  & 0.807  & 0.838  & 0.861  & 0.823  & 0.552  & 0.807  & 0.842  & 0.864  & 0.822 \\ \hline
\\
 & \multicolumn{5}{c}{Batch Means - FT} & \multicolumn{5}{c}{Overlapping Batch Means - FT} \\ \hline

1e4  & 0.461  & 0.682  & 0.767  & 0.737  & 0.737  & 0.458  & 0.677  & 0.765  & 0.733  & 0.721 \\
5e4  & 0.537  & 0.790  & 0.854  & 0.849  & 0.844  & 0.536  & 0.788  & 0.854  & 0.859  & 0.853 \\ 
1e5  & 0.557  & 0.818  & 0.851  & 0.851  & 0.848  & 0.559  & 0.824  & 0.851  & 0.851  & 0.841 \\ 
2e5  & 0.630  & 0.855  & 0.872  & 0.879  & 0.862  & 0.628  & 0.854  & 0.879  & 0.884  & 0.862 \\ \hline
\end{tabular}
\caption{Coverage probabilities for $90\%$ confidence regions over 1000 replications for Bayesian dynamic space-time example.}
\label{tb:coverage_spatial}
\end{center}
\end{table}





\section{Discussion} \label{sec:discussion}
This paper provides theoretical evidence and practical guidance for optimal batch size selection in MCMC simulations. Estimators with the proposed optimal batch sizes are shown to have superior performance versus conventional batch sizes. Batch size selection has not been carefully addressed in multivariate MCMC settings even though sampling multivariate posteriors is routine in Bayesian analyses.  

To reduce computational effort, we used a pilot run length of $1e4$ regardless of the total chain length. Performance of the estimators can be improved by longer pilot runs. This choice was a compromise between computation effort and accuracy. Since practitioners often do not use pilot runs, we repeated the Bayesian logistic regression and Bayesian dynamic space-time model simulations without a pilot run (results not shown). That is, we estimate batch sizes and then estimate $\Sigma$, all with the same MCMC sample.  This improves the coverage probabilities from Tables~\ref{tb:coverage_eel} and~\ref{tb:coverage_spatial} almost universally. However the batch sizes are now random and thus the resulting estimators require separate theoretical analyses.

We study three competing methods of estimating the batch sizes, the AR$(m)$-fit, the nonparametric method, and the lag-based method.  All three methods improve upon current batch size selection methods.  However, the lag-based method does not satisfy Assumption~\ref{ass:batch} and will generally not yield a consistent estimator.  Further, both the nonparametric method and the lag-based methods have high variability and are more computationally intensive than the AR$(m)$-fit. Thus, we recommend using the stable and fast AR$(m)$-fit and intend to make this the default in the \texttt{mcmcse} package.

\begin{appendix}
\section*{Appendix}

\section{Proof of Theorem~\ref{thm:mse_obm} } \label{sec:optimal}

This section presents a proof of Theorem~\ref{thm:mse_obm} and the optimal batch size results for the generalized OBM estimators. We will use the fact that
\[
\text{MSE}\left(\hat{\Sigma}_{w, ij} \right) = \left(\text{Bias}(\hat{\Sigma}_{w,ij})  \right)^2 + \Var\left( \hat{\Sigma}_{w, ij}  \right)\,.
\]
The bias and variance results of the generalized OBM estimators are important in their own right, and are presented here separately.  Denote $\lim_{n\rightarrow \infty} f(n)/g(n)=0$ by $f(n)=o(g(n))$. 
\begin{theorem}
	\label{thm:bias_sv}
Suppose $\text{E}_F\|g\|^{2+\delta} < \infty$ for $\delta > 0$ and  $\sum_{k=1}^{b}k\Delta_{2}w_{n}(k)=1$. If $\{X_t\}$ is a polynomially ergodic Markov chain of order $\xi > (2 + \epsilon)(1 + 2/\delta)$ for some $\epsilon > 0$, then 
$$ \text{Bias}\left(\hat{\Sigma}_{w,ij} \right) = \sum_{k=1}^{b}\Delta_2 w_n(k) \Gamma_{ij} + o\left( \dfrac{b}{n} \right) + o\left(\dfrac{1}{b} \right)\,.$$ 
\end{theorem}
\begin{proof}
By \citet[Theorem 2]{vats:fleg:2018}, $|\Gamma_{ij}| < \infty$.  Then under Assumption~\ref{ass:batch}, for all $i$ and $j$,
\begin{equation}\label{eq:biascova}
\Cov[\bar{Y}_{l}^{(i)}(k),\bar{Y}_{l}^{(j)}(k)]-\Cov[\bar{Y}^{(i)},\bar{Y}^{(j)}]=\dfrac{n-k}{kn}\left(\Sigma_{ij}+\dfrac{n+k}{kn}\Gamma_{ij}+o\left(\dfrac{1}{k^2}\right)\right).
\end{equation} 
Since $\sum_{k=1}^{b}k\Delta_{2}w_{n}(k)=1$, by \eqref{eq:biascova}, 
\begin{align*}
\E\left(\hat{\Sigma}_{w,ij}\right)&=\sum_{k=1}^{b}\dfrac{(n-k+1)(n-k)k\Delta_{2}w_{n}(k)}{n^{2}}\cdot\Sigma_{ij} \\
& \quad \quad +\sum_{k=1}^{b}\dfrac{(n-k+1)(n^{2}-k^{2})\Delta_{2}w_{n}(k)}{n^{3}}\cdot\Gamma_{ij}+o\left(\dfrac{b}{n}\right)+o\left(\dfrac{1}{b}\right)\\
&=\Sigma_{ij}+\sum_{k=1}^{b}\dfrac{(n-k+1)(n^{2}-k^{2})\Delta_{2}w_{n}(k)}{n^{3}}\cdot\Gamma_{ij}+o\left(\dfrac{b}{n}\right)+o\left(\dfrac{1}{b}\right)\\
&=\Sigma_{ij}+\sum_{k=1}^{b}\Delta_{2}w_{n}(k)\cdot\Gamma_{ij}+o\left(\dfrac{b}{n}\right)+o\left(\dfrac{1}{b}\right)\,.
\end{align*}
\end{proof}

Next, we obtain $\Var\left( \hat{\Sigma}_{w, ij}  \right)$.  The proof is under a more general strong invariance principle.  For a function $f: \mathsf{X} \to \mathbb{R}^p$,  assume there exists a $p\times p$ lower triangular matrix $L$, a non-negative increasing function $\psi$ on the positive integers, a finite random variable $D$, and a sufficiently rich probability space $\Omega$ such that for almost all $\omega\in \Omega$ and for all $n>n_{0}$,
\begin{equation}\label{eq:sip}
\left\lVert\sum_{t=1}^{n} f(X) - n \E_F f - LB(n)\right\rVert<D(\omega)\psi(n) \quad \text{with probability 1.}
\end{equation}
Under the conditions of Lemma~\ref{lem:sip}, \cite{vats:fleg:jone:2015spec} establish \eqref{eq:sip} with $\psi(n) = n^{1/2 - \lambda}$ for $\lambda >0$.  Appendix~\ref{ap:prelim} contains a number of preliminary results, followed by the proof of Theorem~\ref{thm:var} in Appendix~\ref{proof theor var}.
\begin{theorem}\label{thm:var}
Suppose \eqref{eq:sip} holds for $f = g$ and $f = g^2$ (where the square is element-wise) such that $\E_FD^4 < \infty$ and Assumption~\ref{ass:batch} holds.  If
\begin{enumerate}
\item $
\sum_{k=1}^{b}(\Delta_{2}w_{k})^2 =  O\left(1/b^2\right),
$
\item $b\psi^{2}(n)\log n(\sum_{k=1}^{b}|\Delta_{2}w_{n}(k)|)^{2}\rightarrow 0$, \text{ and}
\item $\psi^{2}(n)\sum_{k=1}^{b}|\Delta_{2}w_{n}(k)|\rightarrow 0$, then,
\end{enumerate}
\begin{align*}
& \Var \left(\hat{\Sigma}_{w,ij}\right)\\
&=[\Sigma_{ii}\Sigma_{jj}+\Sigma_{ij}^{2}]  \bigg[\dfrac{2}{3}\sum_{k=1}^{b}(\Delta_{2}w_{k})^{2}k^{3} \dfrac{1}{n}+2\sum_{t=1}^{b-1}\sum_{u=1}^{b-t}\Delta_{2}w_{u}  \Delta_{2}w_{t+u} \left(\dfrac{2}{3}u^{3}+u^{2}t\right) \dfrac{1}{n}\bigg]+o\left(\dfrac{b}{n}\right)\notag\\
&:=\left([\Sigma_{ii}\Sigma_{jj}+\Sigma_{ij}^{2}] S \dfrac{b}{n}+o(1) \right) \dfrac{b}{n}.
\end{align*}
\end{theorem}

\section{Preliminaries} \label{ap:prelim}
\begin{proposition}\label{prop:normal2}
If variable $X$ and $Y$ are jointly normally distributed with
$$\left[ \begin{matrix} X \\ Y \end{matrix} \right] \sim N\left( \begin{bmatrix} 0\\ 0\end{bmatrix},\ \left[\begin{matrix}l_{11}& l_{12}\\ l_{12}& l_{22}\end{matrix}\right]\right),$$
then $E[X^{2}Y^{2}]=2l_{12}^{2}+l_{11}l_{22}$.
\end{proposition}
\begin{proposition}\label{prop:normal4}
\citep{jans:pete:stoi:petr:1987} If $X_{1}$, $X_{2}$, $X_{3}$, and $X_{4}$ are jointly normally distributed with mean 0, then
$$E[X_{1}X_{2}X_{3}X_{4}]=E[X_{1}X_{2}]E[X_{3}X_{4}]+E[X_{1}X_{3}]E[X_{2}X_{4}]+E[X_{1}X_{4}]E[X_{2}X_{3}].$$
\end{proposition}

Recall $B=\{B(t),t\geq 0\}$ is a $p$-dimensional standard Brownian motion. Let $B^{(i)}(t)$ be the $i$th component of vector $B(t)$. Denote $\bar{B}=n^{-1}B(n)$, $\bar{B}_{l}(k)=k^{-1}[B(l+k)-B(l)]$.  Let $\Sigma = LL^T$, where $L$ is a lower triangular matrix. Let $C(t)=LB(t)$ and $C^{(i)}(t)$ be the $i$th component of $C(t)$. Suppose $\bar{C}^{(i)}_{l}(k)=k^{-1}(C^{(i)}(l+k)-C^{(i)}(l))$, and $\bar{C}^{(i)}=n^{-1}C^{(i)}(n)$.

We now present some specific preliminary results and notation for the proof of Theorem~\ref{thm:var}.  For $0 < c_2 < c_1 < 1$, let
\begin{equation*}
A_{2}=\dfrac{(c_1b)^{2}}{n^{2}}\E\left[\left[\sum_{l=0}^{n-c_1b} \left(\bar{C}_{l}^{(i)}(c_1b)-\bar{C}^{(i)} \right) \left(\bar{C}_{l}^{(j)}(c_1b)-\bar{C}^{(j)} \right)\right]^{2}\right], \text{ and}	
\end{equation*}
\begin{equation*}
A_3 = -\dfrac{c_1c_2b^{2}}{n^{2}} \E \left[\sum_{l=0}^{n-c_1b}(\bar{C}_{l}^{(i)}(c_1b)-\bar{C}^{(i)})(\bar{C}_{l}^{(j)}(c_1b)-\bar{C}^{(j)})\right]\left[\sum_{l=0}^{n-c_2b}(\bar{C}_{l}^{(i)}(c_2b)-\bar{C}^{(i)})(\bar{C}_{l}^{(j)}(c_2b)-\bar{C}^{(j)})\right].
\end{equation*}

\begin{lemma}
	\label{lemma:A2_var_term}
For $0 < c_2 < c_1 < 1$, 
\begin{equation*}
A_{2} = \left[\dfrac{2}{3}(\Sigma_{ii}\Sigma_{jj}+\Sigma_{ij}^{2})\cdot\dfrac{c_1b}{n}+\Sigma_{ij}^{2}-4\Sigma_{ij}^{2}\cdot\dfrac{c_1b}{n}\right]+o\left(\dfrac{b}{n}\right) \text{ and },
\end{equation*}
\begin{equation*}
A_{3} = \dfrac{(c_2 - 3c_1)c_2}{3c_1}(\Sigma_{ii}\Sigma_{jj}  + \Sigma_{ij}^{2})\cdot \dfrac{b}{n} + 2\left(c_1+c_2\right)\cdot\Sigma_{ij}^{2}\cdot\dfrac{b}{n} - \Sigma_{ij}^{2} +o\left(\dfrac{b}{n}\right)
\end{equation*}
\end{lemma}

\begin{proof}
Denote
\begin{equation}\label{eq:a1}
a_1=\sum_{l=0}^{n-c_1b} \left(\bar{C}_{l}^{(i)}(c_1b)-\bar{C}^{(i)} \right)^{2} \left(\bar{C}_{l}^{(j)}(c_1b)-\bar{C}^{(j)} \right)^{2},
\end{equation}
\begin{equation}\label{eq:a2}
a_2=\sum_{s=1}^{c_1b-1}\sum_{l=0}^{n-c_1b-s}(\bar{C}_{l}^{(i)}(c_1b)-\bar{C}^{(i)})(\bar{C}_{l}^{(j)}(c_1b)-\bar{C}^{(j)})(\bar{C}^{(i)}_{l+s}(c_1b)-\bar{C}^{(i)})(\bar{C}^{(j)}_{l+s}(c_1b)-\bar{C}^{(j)}),
\end{equation}
\begin{equation*}\label{eq:a3}
a_3=\sum_{s=b}^{n-c_1b}\sum_{l=0}^{n-c_1b-s} \left(\bar{C}_{l}^{(i)}(c_1b)-\bar{C}^{(i)} \right) \left(\bar{C}_{l}^{(j)}(c_1b)-\bar{C}^{(j)} \right) \left(\bar{C}^{(i)}_{l+s}(c_1b)-\bar{C}^{(i)} \right) \left(\bar{C}^{(j)}_{l+s}(c_1b)-\bar{C}^{(j)} \right).
\end{equation*}
Then
\begin{align}\label{eq:A1express}
A_{1}&=\dfrac{c_1b^{2}}{n^{2}}\E \left[\sum_{l=0}^{n-c_1b} \left(\bar{C}_{l}^{(i)}(c_1b)-\bar{C}^{(i)} \right)^{2} \left(\bar{C}_{l}^{(j)}(c_1b)-\bar{C}^{(j)} \right)^{2}\right.\nonumber\\
&+2\sum_{s=1}^{c_1b-1}\sum_{l=0}^{n-c_1b-s}(\bar{C}_{l}^{(i)}(c_1b)-\bar{C}^{(i)})(\bar{C}_{l}^{(j)}(c_1b)-\bar{C}^{(j)})(\bar{C}^{(i)}_{l+s}(c_1b)-\bar{C}^{(i)})(\bar{C}^{(j)}_{l+s}(c_1b)-\bar{C}^{(j)})\nonumber\\
&+\left.2\sum_{s=b}^{n-c_1b}\sum_{l=0}^{n-c_1b-s} \left(\bar{C}_{l}^{(i)}(c_1b)-\bar{C}^{(i)} \right) \left(\bar{C}_{l}^{(j)}(c_1b)-\bar{C}^{(j)} \right) \left(\bar{C}^{(i)}_{l+s}(c_1b)-\bar{C}^{(i)} \right) \left(\bar{C}^{(j)}_{l+s}(c_1b)-\bar{C}^{(j)} \right) \right].\nonumber\\
&=\dfrac{c_1b^{2}}{n^{2}}E[a_1+2a_2+2a_3].
\end{align}
First we calculate $E[a_1]$ at \eqref{eq:a1}. Let $U_{t}^{(i)}=B^{(i)}(t)-B^{(i)}(t-1)$, then $U_{t}^{(i)}\distas{iid}N(0,1)$ for $t=1,2,...,n$ and
$$\bar{B}_{l}^{(i)}(c_1b)-\bar{B}^{(i)}=\dfrac{(n-c_1b)}{nc_1b}\sum_{t=l+1}^{l + c_1b}U^{(i)}_{t}-\dfrac{1}{n}\sum_{t=1}^{l}U_{t}^{(i)}-\dfrac{1}{n}\sum_{t=l+c_1b+1}^{n}U_{t}^{(i)}.$$
Notice that $\E[\bar{B}_{l}^{(i)}(c_1b)-\bar{B}^{(i)}]=0$ for $l=0,...,(n-c_1b)$ and 
$$\Var[\bar{B}_{l}^{(i)}(c_1b)-\bar{B}^{(i)}]=\Big(\dfrac{n-c_1b}{nc_1b}\Big)^{2}c_1b+\dfrac{n-c_1b}{n^{2}}=\dfrac{n-c_1b}{c_1bn},$$
therefore
$$\bar{B}^{(i)}_{l}(c_1b)-\bar{B}^{(i)}\sim N\left(0,\ \dfrac{n-c_1b}{c_1bn}\right)$$
and
$$\bar{B}_{l}(c_1b)-\bar{B}_{n}\sim N\left(0,\ \dfrac{n-c_1b}{c_1bn}I_{p}\right),$$
hence
\begin{equation}\label{eq:Cdistr}
\bar{C}_{l}(c_1b)-\bar{C}_{n}=L(\bar{B}_{l}-\bar{B}_{n})\sim N\left(0,\ \dfrac{n-c_1b}{c_1bn}LL^{T}\right).
\end{equation}
Now consider $\E[(\bar{C}_{l}^{(i)}(c_1b)-\bar{C}^{(i)})^{2}(\bar{C}_{l}^{(j)}(c_1b)-\bar{C}^{(j)})^{2}]:= \E[Z_{i}^{2}Z_{j}^{2}]$ where $Z_{i}=\bar{C}_{l}^{(i)}(c_1b)-\bar{C}^{(i)}$ and $Z_{j}=\bar{C}_{l}^{(j)}(c_1b)-\bar{C}^{(j)}$. Recall $\Sigma=LL^{T}$, then
$$\left[ \begin{matrix} Z_{i} \\ Z_{j} \end{matrix} \right] \sim N\left( \begin{bmatrix} 0\\ 0\end{bmatrix},\ \dfrac{n-c_1b}{c_1bn}\left[\begin{matrix}\Sigma_{ii}& \Sigma_{ij}\\ \Sigma_{ij}& \Sigma_{jj}\end{matrix}\right]\right).$$
Apply Proposition~\ref{prop:normal2},
\begin{align}\label{eq:a1eq1}
 \E\left[ \left(\bar{C}_{l}^{(i)}(c_1b)-\bar{C}^{(i)} \right)^{2} \left(\bar{C}_{l}^{(j)}(c_1b)-\bar{C}^{(j)} \right)^{2} \right] &=2\left(\dfrac{n-c_1b}{c_1bn}\Sigma_{ij}\right)^{2}+\left(\dfrac{n-c_1b}{c_1bn}\Sigma_{ii}\right)\left(\dfrac{n-c_1b}{c_1bn}\Sigma_{jj}\right)\nonumber\\
&=\left(\dfrac{n-c_1b}{c_1bn}\right)^{2}[\Sigma_{ij}^{2}+\Sigma_{ii}\Sigma_{jj}]+\left(\dfrac{n-c_1b}{c_1bn}\right)^{2}\Sigma^{2}_{ij}.
\end{align}
Replace \eqref{eq:a1eq1} in \eqref{eq:a1}
\begin{align}\label{eq:a1expressionfinal}
\E[a_1]&=\sum_{l=0}^{n-c_1b}\E \left[ \left(\bar{C}_{l}^{(i)}(c_1b)-\bar{C}^{(i)} \right)^2 \left(\bar{C}_{l}^{(j)}(c_1b)-\bar{C}^{(j)} \right)^2 \right]\nonumber\\
&=(n-c_1b+1)\left(\dfrac{n-c_1b}{c_1bn}\right)^{2}(\Sigma_{ij}^{2}+\Sigma_{ii}\Sigma_{jj})+\sum_{l=0}^{n-c_1b}\left(\dfrac{n-c_1b}{c_1bn}\right)^{2}\Sigma_{ij}^{2}\nonumber\\
&=\sum_{l=0}^{n-c_1b}\left(\dfrac{n-c_1b}{c_1bn}\right)^{2}\Sigma_{ij}^{2}+o\left(\dfrac{n}{b}\right).
\end{align}
To calculate $\E[a_2]$ for $s=1,2,..., (c_1b-1)$, we require
\[
\E \left[ \left(\bar{C}_{l}^{(i)}(c_1b)-\bar{C}^{(i)} \right) \left(\bar{C}_{l}^{(j)}(c_1b)-\bar{C}^{(j)} \right) \left(\bar{C}^{(i)}_{l+s}(c_1b)-\bar{C}^{(i)} \right) \left(\bar{C}^{(j)}_{l+s}(c_1b)-\bar{C}^{(j)} \right) \right]\,.
\]
%
Notice that 
\begin{align*}
 \text{Cov}(\bar{C}_{l}(c_1b)-\bar{C}_n,\ \bar{C}_{l+s}(c_1b)-\bar{C}_n)\nonumber &=\E \left[ \left(\bar{C}_{l}(c_1b)-\bar{C}_n \right) \left(\bar{C}_{l+s}(c_1b)-\bar{C}_n \right)^{T} \right]\nonumber\\
&=L\cdot \E \left[ \left(\bar{B}_{l}(c_1b)-\bar{B}_n \right) \left(\bar{B}_{l+s}(c_1b)-\bar{B}_n \right)^{T} \right]\cdot L^{T}.
\end{align*}
Consider each entry of $E[(\bar{B}_{l}(c_1b)-\bar{B})(\bar{B}_{l+s}(c_1b)-\bar{B})^{T}]$.
For $i\neq j$,
\begin{equation}
\E\left[\bar{B}^{(i)}_{l}(c_1b)-\bar{B}^{(i)} \right] \left[\bar{B}^{(j)}_{l+s}(c_1b)-\bar{B}^{(j)} \right] = \E \left[\bar{B}^{(i)}_{l}(c_1b)-\bar{B}^{(i)} \right]\cdot \E \left[\bar{B}^{(j)}_{l+s}(c_1b)-\bar{B}^{(j)} \right]=0. \label{eq:Ei!=j}
\end{equation}
For $i=j$, we require $\E[\bar{B}^{(i)}_{l}(c_1b)-\bar{B}^{(i)}][\bar{B}^{(i)}_{l+s}(c_1b)-\bar{B}^{(i)}]$. 
\begin{align*}
&\E \left[\bar{B}^{(i)}_{l}(c_1b)-\bar{B}^{(i)} \right] \left[\bar{B}^{(i)}_{l+s}(c_1b)-\bar{B}^{(i)} \right] \\
& =  \E \left[\bar{B}^{(i)}_l(c_1b) \bar{B}^{(i)}_{l+s}(c_1b)  \right] + \E \left[\bar{B}^{(i)} \bar{B}^{(i)} \right] - \E \left[ \bar{B}^{(i)} \bar{B}^{(i)}_{l+s}(c_1b)  \right] - \E \left[ \bar{B}^{(i)} \bar{B}^{(i)}_{l}(c_1b)  \right] \\ 
& = \dfrac{1}{c_1b^2} \E\left[ \left(B^{(i)}(l+c_1b) - B^{(i)}(l) \right) \left(B^{(i)}(l+s+c_1b) - B^{(i)}(l + s) \right) \right] \\
& \quad + \dfrac{1}{n^2}\E\left[ \left(B^{(i)}(n) \right)^2 \right] - \dfrac{1}{nc_1b}\E\left[B^{(i)}(n) \left( B^{(i)} (l+c_1b+s) - B^{(i)}(l+s)\right) \right]\\ 
& \quad - \dfrac{1}{nc_1b}\E\left[B^{(i)}(n) \left( B^{(i)} (l+c_1b) - B^{(i)}(l)\right) \right] \\ 
& = \dfrac{c_1b - s}{c_1^2b^2} + \dfrac{1}{n} - \dfrac{2}{n}\\ & = \dfrac{n - c_1b}{c_1bn} - \dfrac{s}{n^2}\,. \numberthis \label{eq:Ei=j}
\end{align*}

Combine \eqref{eq:Ei=j} and \eqref{eq:Ei!=j},
\begin{equation*}
\text{Cov} \left(\bar{C}_{l}(c_1b)-\bar{C}_n,\ \bar{C}_{l+s}(c_1b)-\bar{C}_n \right)=L\cdot\left(\dfrac{n-c_1b}{c_1bn}-\dfrac{s}{c_1^2b^{2}}\right)I_{p}\cdot L^{T}=\left(\dfrac{n-c_1b}{c_1bn}-\dfrac{s}{c_1^2b^{2}}\right)\cdot\Sigma.
\end{equation*} 
Given \eqref{eq:Cdistr}, \eqref{eq:Ei!=j} and \eqref{eq:Ei=j} and let $Z_{1}=\bar{C}_{l}(c_1b)^{(i)}-\bar{C}^{(i)}_n$, $Z_{2}=\bar{C}_{l}(c_1b)^{(j)}-\bar{C}^{(j)}_n$, $Z_{3}=\bar{C}_{l+s}(c_1b)^{(i)}-\bar{C}^{(i)}_n$, $Z_{4}=\bar{C}_{l+s}(c_1b)^{(j)}-\bar{C}^{(j)}_n$, $(Z_1, Z_2, Z_3, Z_4)^T$ has a $4$-dimensional Normal distribution with mean 0, and covariance matrix,
$$  \left[\begin{matrix}\left(\dfrac{n-c_1b}{c_1bn}\right)\Sigma_{ii}& \left(\dfrac{n-c_1b}{c_1bn}\right)\Sigma_{ij} &\left(\dfrac{n-c_1b}{c_1bn}-\dfrac{s}{c_1^2b^{2}}\right)\Sigma_{ii}&\left(\dfrac{n-c_1b}{c_1bn}-\dfrac{s}{c_1^2b^{2}}\right)\Sigma_{ij}\\ & \left(\dfrac{n-c_1b}{c_1bn}\right)\Sigma_{jj}& \left(\dfrac{n-c_1b}{c_1bn}-\dfrac{s}{c_1^2b^{2}}\right)\Sigma_{ij}&\left(\dfrac{n-c_1b}{c_1bn}-\dfrac{s}{c_1^2b^{2}}\right)\Sigma_{jj}\\ &  &\left(\dfrac{n-c_1b}{c_1bn}\right)\Sigma_{ii}&\left(\dfrac{n-c_1b}{c_1bn}\right)\Sigma_{ij} \\&&&\left(\dfrac{n-c_1b}{c_1bn}\right)\Sigma_{jj}\end{matrix}\right].$$
Only upper triangle entries are presented due to symmetry of the matrix. By Proposition~\ref{prop:normal4},
\begin{align*}
& \E[Z_{1}Z_{2}Z_{3}Z_{4}]\\
&= \E[Z_{1}Z_{2}]\cdot \E[Z_{3}Z_{4}] + \E[Z_{1}Z_{3}]\cdot\E[Z_{2}Z_{4}]+ \E[Z_{1}Z_{4}]\cdot\E[Z_{2}Z_{3}]\nonumber\\
&=\left(\dfrac{n-c_1b}{c_1bn}\right)^{2}\Sigma_{ij}^{2}+\left(\dfrac{n-c_1b}{c_1bn}-\dfrac{s}{c_1^2b^{2}}\right)^{2}\Sigma_{ii}\Sigma_{jj}+\left(\dfrac{n-c_1b}{c_1bn}-\dfrac{s}{c_1^2b^{2}}\right)^{2}\Sigma_{ij}^{2} \numberthis \label{eq:Zdistr4}\,.
\end{align*}
Plug \eqref{eq:Zdistr4} in \eqref{eq:a2},
\begin{align}\label{eq:a2expression1}
&\E[a_2] \\ 
&=\sum_{s=1}^{c_1b-1}\sum_{l=0}^{n-c_1b-s} \E \left[ \left(\bar{C}_{l}^{(i)}(c_1b)-\bar{C}^{(i)} \right) \left(\bar{C}_{l}^{(j)}(c_1b)-\bar{C}^{(j)} \right) \left(\bar{C}^{(i)}_{l+s}(c_1b)-\bar{C}^{(i)} \right) \left(\bar{C}^{(j)}_{l+s}(c_1b)-\bar{C}^{(j)} \right) \right]\nonumber\\
&=\sum_{s=1}^{c_1b-1}\sum_{l=0}^{n-c_1b-s} \E[Z_{1}Z_{2}Z_{3}Z_{4}]\nonumber\\
&=\sum_{s=1}^{c_1b-1}\sum_{l=0}^{n-c_1b-s}\left[\left(\dfrac{n-c_1b}{c_1bn}\right)^{2}\Sigma_{ij}^{2}+\left(\dfrac{n-c_1b}{c_1bn}-\dfrac{s}{c_1^2b^{2}}\right)^{2}\Sigma_{ii}\Sigma_{jj}+\left(\dfrac{n-c_1b}{c_1bn}-\dfrac{s}{c_1^2b^{2}}\right)^{2}\Sigma_{ij}^{2}\right].\nonumber
\end{align}
Notice that
\begin{align}\label{eq:a2expression2}
&\sum_{s=1}^{c_1b-1}\sum_{l=0}^{n-c_1b-s}\left(\dfrac{n-c_1b}{c_1bn}-\dfrac{s}{c_1^2b^{2}}\right)^{2}\nonumber\\
&=\sum_{s=1}^{c_1b-1}\sum_{l=0}^{n-c_1b-s}\left[\dfrac{s^{2}}{c_1^4b^{4}}+\left(\dfrac{2}{c_1^2b^{2}n}-\dfrac{2}{c_1^3b^{3}}\right)s+\left(\dfrac{1}{c_1^2b^{2}}+\dfrac{1}{n^{2}}-\dfrac{2}{c_1bn}\right)\right]\nonumber\\
&=\sum_{s=1}^{c_1b-1}\left[-\dfrac{s^{3}}{c_1^4b^{4}}+\left(\dfrac{n}{c_1^4b^{4}}+\dfrac{1}{c_1^3b^{3}}+\dfrac{1}{c_1^4b^{4}}-\dfrac{2}{c_1^2b^{2}n}\right)s^{2}\right.\nonumber\\
&\ \  \ \ \ \ \ \ \ \ \ +\left(\dfrac{3}{c_1^2b^{2}}-\dfrac{2n}{c_1^3b^{3}}+\dfrac{2}{c_1^2b^{2}n}-\dfrac{2}{c_1^3b^{3}}-\dfrac{1}{n^{2}}\right)s\nonumber\\
&\ \  \ \ \ \ \ \ \ \ \ \left.+\left(\dfrac{n}{c_1^2b^{2}}+\dfrac{3}{n}-\dfrac{3}{c_1b}+\dfrac{1}{c_1^2b^{2}}+\dfrac{1}{n^{2}}-\dfrac{2}{c_1bn}-\dfrac{c_1b}{n^{2}}\right)\right]\nonumber\\
&=-\dfrac{1}{c_1^4b^{4}}\left(\dfrac{c_1^4b^{4}}{4}-\dfrac{c_1^3b^{3}}{2}+\dfrac{c_1^2b^{2}}{4}\right)+\left(\dfrac{n}{c_1^4b^{4}}+\dfrac{1}{c_1^3b^{3}}+\dfrac{1}{c_1^4b^{4}}-\dfrac{2}{c_1^2b^{2}n}\right)\left(\dfrac{c_1^3b^{3}}{3}-\dfrac{c_1^2b^{2}}{2}+\dfrac{c_1b}{6}\right)\nonumber\\
&\ \ \ \ \ +\left(\dfrac{3}{b^{2}}-\dfrac{2n}{b^{3}}+\dfrac{2}{b^{2}n}-\dfrac{2}{b^{3}}-\dfrac{1}{n^{2}}\right)\left(\dfrac{b^{2}}{2}-\dfrac{b}{2}\right)\nonumber\\
&\ \ \ \ \ +\left(\dfrac{n}{c_1^2b^{2}}+\dfrac{3}{n}-\dfrac{3}{c_1b}+\dfrac{1}{c_1^2b^{2}}+\dfrac{1}{n^{2}}-\dfrac{2}{c_1bn}-\dfrac{c_1b}{n^{2}}\right)\left(c_1b-1\right)\nonumber\\
&=\dfrac{n}{c_1^4b^{4}}\cdot\dfrac{c_1^3b^{3}}{3}-\dfrac{2n}{c_1^3b^{3}}\cdot\dfrac{c_1^2b^{2}}{2}+\dfrac{n}{c_1^2b^{2}}\cdot c_1b\nonumber\\
&=\dfrac{1}{3}\dfrac{n}{c_1b}+o\left(\dfrac{n}{b}\right).
\end{align}
Plug \eqref{eq:a2expression2} in \eqref{eq:a2expression1}
\begin{align}\label{eq:a2expressionfinal}
E[a_2]&=\Sigma_{ij}^{2}\sum_{s=1}^{c_1b-1}\sum_{l=0}^{n-c_1b-s}\left(\dfrac{n-c_1b}{c_1bn}\right)^{2}+(\Sigma_{ii}\Sigma_{jj}+\Sigma_{ij}^{2})\left[\dfrac{1}{3}\dfrac{n}{c_1b}+o\left(\dfrac{n}{b}\right)\right].
\end{align}
Similarly as $E[a_2]$, we calculate $E[a_3]$ by first calculating
\[
\E\left[ \left(\bar{C}_{l}^{(i)}(c_1b)-\bar{C}^{(i)} \right) \left(\bar{C}_{l}^{(j)}(c_1b)-\bar{C}^{(j)} \right) \left(\bar{C}^{(i)}_{l+s}(c_1b)-\bar{C}^{(i)} \right) \left(\bar{C}^{(j)}_{l+s}(c_1b)-\bar{C}^{(j)} \right) \right]\,,
\]
for $s=c_1b,..., (n-c_1b).$ We will show that
\begin{equation}\label{eq:jointdis2pa3}
\left[ \begin{matrix} \bar{C}_{l}(b)-\bar{C} \\ \bar{C}_{l+s}(b)-\bar{C} \end{matrix} \right] \sim N\left( \begin{bmatrix} 0\\ 0\end{bmatrix},\ \left[\begin{matrix}\left(\dfrac{n-b}{bn}\right)\Sigma& -\dfrac{1}{n}\Sigma\\ -\dfrac{1}{n}\Sigma& \left(\dfrac{n-b}{bn}\right)\Sigma\end{matrix}\right]\right).
\end{equation}
Continuing as in \eqref{eq:Ei!=j}
\begin{equation}\label{eq:Ccova3}
\text{Cov} \left(\bar{C}_{l}(c_1b)-\bar{C},\ \bar{C}_{l+s}(c_1b)-\bar{C} \right)=L\cdot\left(-\dfrac{1}{n}\right)I_{p}\cdot L^{T}=-\dfrac{1}{n}\cdot\Sigma.
\end{equation}
The joint distribution in \eqref{eq:jointdis2pa3} follows \eqref{eq:Ccova3} and \eqref{eq:Cdistr}. Denote $Z_{1}=\bar{C}_{l}(c_1b)^{(i)}-\bar{C}^{(i)}$, $Z_{2}=\bar{C}_{l}(c_1b)^{(j)}-\bar{C}^{(j)}$, $Z_{3}=\bar{C}_{l+s}(c_1b)^{(i)}-\bar{C}^{(i)}$, $Z_{4}=\bar{C}_{l+s}(c_1b)^{(j)}-\bar{C}^{(j)}$. By Proposition~\ref{prop:normal4}, 
\begin{align}\label{eq:a3expression}
& \E[a_3]\\
&=\sum_{s=c_1b}^{n-c_1b}\sum_{l=0}^{n-c_1b-s} \E \left[ \left(\bar{C}_{l}^{(i)}(c_1b)-\bar{C}^{(i)} \right) \left(\bar{C}_{l}^{(j)}(c_1b)-\bar{C}^{(j)} \right)  \left(\bar{C}^{(i)}_{l+s}(c_1b)-\bar{C}^{(i)} \right) \left(\bar{C}^{(j)}_{l+s}(c_1b)-\bar{C}^{(j)} \right) \right]\nonumber\\
&=\sum_{s=b}^{n-c_1b}\sum_{l=0}^{n-c_1b-s} \E[Z_{1}Z_{2}Z_{3}Z_{4}]\nonumber\\
&=\sum_{s=b}^{n-c_1b}\sum_{l=0}^{n-c_1b-s}\left[\left(\dfrac{n-c_1b}{c_1bn}\right)^{2}\Sigma_{ij}^{2}+\dfrac{1}{n^{2}}\Sigma_{ii}\Sigma_{jj}+\dfrac{1}{n^{2}}\Sigma_{ij}^{2}\right]. \nonumber
\end{align}
Notice
\begin{align}\label{eq:a3expression1}
\sum_{s=c_1b}^{n-c_1b}\sum_{l=1}^{n-c_1b+1-s}\dfrac{1}{n^{2}}\nonumber & =\dfrac{1}{n^{2}}\cdot\sum_{s=c_1b}^{n-c_1b}(n-c_1b-s+1)\nonumber\\
&=-\dfrac{1}{n^{2}}\left(\dfrac{n^{2}}{2}-c_1bn+\dfrac{n}{2}\right)+\left(\dfrac{1}{n}-\dfrac{c_1b}{n^{2}}+\dfrac{1}{n^{2}}\right)(n-2c_1b+1)\nonumber\\
&=o\left(\dfrac{n}{b}\right).
\end{align}
Plug \eqref{eq:a3expression1} in \eqref{eq:a3expression}, 
\begin{align}\label{eq:a3expressionfinal}
\E[a_3]&=\Sigma_{ij}^{2}\sum_{s=b}^{n-c_1b}\sum_{l=0}^{n-c_1b-s}\left(\dfrac{n-c_1b}{c_1bn}\right)^{2}+o\left(\dfrac{n}{b}\right).
\end{align}
Plug \eqref{eq:a1expressionfinal},\eqref{eq:a2expressionfinal} and \eqref{eq:a3expressionfinal} in \eqref{eq:A1express},
\begin{align*}
A_{2}& = \E\left[\dfrac{c_1^2b^{2}}{n^{2}}\left[\sum_{l=0}^{n-c_1b}\left(\bar{C}_{l}^{(i)}(c_1b)-\bar{C}^{(i)} \right)  \left(\bar{C}_{l}^{(j)}(c_1b)-\bar{C}^{(j)} \right)\right]^{2}\right]\notag\\
& = \dfrac{c_1^2b^{2}}{n^{2}}\cdot[\E a_1 + 2\E a_2 + 2\E a_3]\notag\\
& = \dfrac{c_1b^{2}}{n^{2}}\cdot\Bigg[\dfrac{2}{3}(\Sigma_{ii}\Sigma_{jj}+\Sigma_{ij}^{2})\cdot\dfrac{n}{c_1b}+o\left(\dfrac{n}{b}\right)\notag\\
& + \Sigma_{ij}^{2}\left(\sum_{l=0}^{n-c_1b}\left(\dfrac{n-c_1b}{c_1bn}\right)^{2}+2\sum_{s=1}^{c_1b-1}\sum_{l=0}^{n-c_1b-s}\left(\dfrac{n-c_1b}{c_1bn}\right)^{2}+2\sum_{s=c_1b}^{n-c_1b}\sum_{l=0}^{n-c_1b-s}\left(\dfrac{n-c_1b}{c_1bn}\right)^{2}\right)\Bigg]\notag\\
&=\dfrac{c_1^2b^{2}}{n^{2}}\cdot\left[\dfrac{2}{3}(\Sigma_{ii}\Sigma_{jj}+\Sigma_{ij}^{2})\cdot\dfrac{n}{c_1b}+o\left(\dfrac{n}{b}\right)+\Sigma_{ij}^{2}\left(\dfrac{n-c_1b}{c_1bn}\right)^{2}(n-c_1b+1)^{2}\right]\notag\\
&= \dfrac{c_1^2b^{2}}{n^{2}}\cdot\left[\dfrac{2}{3}(\Sigma_{ii}\Sigma_{jj}+\Sigma_{ij}^{2})\cdot\dfrac{n}{c_1b}+o\left(\dfrac{n}{c_1b}\right)+\Sigma_{ij}^{2}\left(\dfrac{n^{2}}{c_1^2b^{2}}-\dfrac{4n}{c_1b}\right)\right]\notag\\
&=\left[\dfrac{2}{3}(\Sigma_{ii}\Sigma_{jj}+\Sigma_{ij}^{2})\cdot\dfrac{c_1b}{n}+\Sigma_{ij}^{2}-4\Sigma_{ij}^{2}\cdot\dfrac{c_1b}{n}\right]+o\left(\dfrac{b}{n}\right).
\end{align*}
That proves the first part of the lemma. We now move on to term $A_3$.
Let
\begin{equation*}
OL^{(i)} = \left(\bar{C}_{p}^{(i)}(c_1b)-\bar{C}^{(i)} \right) \left(\bar{C}_{q}^{(i)}(c_2b)-\bar{C}^{(i)} \right),
\end{equation*}
and
\begin{equation*}
OL^{(j)} = \left(\bar{C}_{p}^{(j)}(c_1b)-\bar{C}^{(j)} \right) \left(\bar{C}_{q}^{(j)}(c_2b)-\bar{C}^{(j)} \right),
\end{equation*}
for $p$, $q$ satisfying $q\geq p$ and $q+c_1b\leq p + c_2b.$ Then
\begin{align}\label{eq:A3expression}
& A_{3}\\
&=- \dfrac{c_1c_2b^{2}}{n^{2}}\left[\sum_{l=0}^{n-c_1b} \left(\bar{C}_{l}^{(i)}(c_1b)-\bar{C}^{(i)} \right) \left(\bar{C}_{l}^{(j)}(c_1b)-\bar{C}^{(j)} \right)\right]     \left[\sum_{l=0}^{n-c_2b}  \left(\bar{C}_{l}^{(i)}(c_2b)-\bar{C}^{(i)} \right)  \left(\bar{C}_{l}^{(j)}(c_2b)-\bar{C}^{(j)} \right) \right]\nonumber\\
&= - \dfrac{c_1c_2b^{2}}{n^{2}}\cdot\E[((c_1-c_2)b+1)(n-b+1)\cdot OL^{(i)}OL^{(j)}\nonumber\\
& \quad +2\sum_{s=1}^{c_2b-1}\sum_{l=0}^{n-c_1b-s}(\bar{C}_{l}(c_1b)^{(i)}-\bar{C}^{(i)})(\bar{C}_{l}(c_1b)^{(j)}-\bar{C}^{(j)})(\bar{C}^{(i)}_{l+(c_1-c_2)b+s}(c_2b)-\bar{C}^{(i)})(\bar{C}^{(j)}_{l+(c_1-c_2)b+s}(c_2b)-\bar{C}^{(j)})\nonumber\\
& \quad +2\sum_{s=c_2b}^{n-c_1b}\sum_{l=0}^{n-c_1b-s}(\bar{C}_{l}(c_1b)^{(i)}-\bar{C}^{(i)})(\bar{C}_{l}(c_1b)^{(j)}-\bar{C}^{(j)})(\bar{C}^{(i)}_{l+(c_1-c_2)b+s}(c_2b)-\bar{C}^{(i)})(\bar{C}^{(j)}_{l+(c_1-c_2)b+s}(c_2b)-\bar{C}^{(j)})], \nonumber
\end{align}
Denote the two double sums in \eqref{eq:A3expression} by: 
$$a_4=\sum_{s=1}^{c_2b-1}\sum_{l=0}^{n-c_1b-s}(\bar{C}_{l}(c_1b)^{(i)}-\bar{C}^{(i)})(\bar{C}_{l}(c_1b)^{(j)}-\bar{C}^{(j)})(\bar{C}^{(i)}_{l+(c_1-c_2)b+s}(c_2b)-\bar{C}^{(i)})(\bar{C}^{(j)}_{l+(c_1-c_2)b+s}(c_2b)-\bar{C}^{(j)}),$$
$$a_5 = \sum_{s=c_2b}^{n-c_1b}\sum_{l=0}^{n-c_1b-s}(\bar{C}_{l}(c_1b)^{(i)}-\bar{C}^{(i)})(\bar{C}_{l}(c_1b)^{(j)}-\bar{C}^{(j)})(\bar{C}^{(i)}_{l+(c_1-c_2)b+s}(c_2b)-\bar{C}^{(i)})(\bar{C}^{(j)}_{l+(c_1-c_2)b+s}(c_2b)-\bar{C}^{(j)})].$$
First consider $\E[OL^{(i)}OL^{(j)}]$ at \eqref{eq:A3expression}. 
We will show that
\begin{equation}\label{eq:jointdistr2pOL}
\left[ \begin{matrix} \bar{C}_{p}(c_1b)-\bar{C}_{n} \\ \bar{C}_{q}(c_2b)-\bar{C}_{n} \end{matrix} \right] \sim N\left( \begin{bmatrix} 0\\ 0\end{bmatrix},\ \left[\begin{matrix}\left(\dfrac{n-c_1b}{c_1bn}\right)\Sigma& \left(\dfrac{n-c_1b}{c_1bn}\right)\Sigma\\ \left(\dfrac{n-c_1b}{c_1bn}\right)\Sigma& \left(\dfrac{n-c_2b}{c_2bn}\right)\Sigma\end{matrix}\right]\right).
\end{equation} 
For $i\neq j$,
\begin{equation}\label{eq:repeat1}
\E \left[\bar{B}^{(i)}_{p}(c_1b)-\bar{B}^{(i)} \right]  \left[\bar{B}^{(j)}_{q}(c_2b)-\bar{B}^{(j)} \right] = \E \left[\bar{B}^{(i)}_{p}(c_1b)-\bar{B}^{(i)} \right]\cdot \E \left[\bar{B}^{(j)}_{q}(c_2b)-\bar{B}^{(j)}\right] = 0.
\end{equation}
For $i=j$ and $p$, $q$ satisfying $q\geq p$ and $q + c_2b\leq p + c_1b$,
following steps similar to \eqref{eq:Ei=j},
\begin{equation}
\label{eq:repeat2}
\E \left[\bar{B}^{(i)}_{p}(c_1b)-\bar{B}^{(i)} \right] \left[\bar{B}^{(i)}_{q}(c_2b)-\bar{B}^{(i)} \right] = \dfrac{n - c_1b}{nc_1b}\,.
\end{equation}
By \eqref{eq:repeat1} and \eqref{eq:repeat2}
\begin{equation}\label{eq:Ceq3a4}
\text{Cov} \left(\bar{C}_{p}(c_1b)-\bar{C}_{n},\ \bar{C}_{q}(c_2b)-\bar{C}_{n} \right)=L\cdot\left(\dfrac{n-c_1b}{c_1bn}\right)I_{p}\cdot L^{T}=\dfrac{n-c_1b}{c_1bn}\cdot\Sigma.
\end{equation}
Equation \eqref{eq:Ceq3a4} yields the joint distribution at \eqref{eq:jointdistr2pOL}. Denote $Z_{1}=\bar{C}_{p}(c_1b)^{(i)}-\bar{C}^{(i)}$, $Z_{2}=\bar{C}_{p}(c_1b)^{(j)}-\bar{C}^{(j)}$, $Z_{3}=\bar{C}_{q}(c_2b)^{(i)}-\bar{C}^{(i)}$, $Z_{4}=\bar{C}_{q}(c_2b)^{(j)}-\bar{C}^{(j)}$. Then
\begin{align}\label{eq:EOLexpression}
& \E \left[((c_1-c_2)b+1)(n-c_1b+1)OL^{(i)}OL^{(j)} \right]\nonumber\\
&=((c_1-c_2)b+1)(n-c_1b+1)\cdot\\ 
& \quad  \times \E \left[(\bar{C}_{p}^{(i)}(c_1b)-\bar{C}^{(i)})(\bar{C}_{p}^{(j)}(c_1b)-\bar{C}^{(j)})(\bar{C}^{(i)}_{q}(c_2b)-\bar{C}^{(i)})(\bar{C}^{(j)}_{q}(c_2b)-\bar{C}^{(j)}) \right]\nonumber\\
&=((c_1-c_2)b+1)(n-c_1b+1)\cdot \E[Z_{1}Z_{2}Z_{3}Z_{4}]\nonumber\\
&=((c_1-c_2)b+1)(n-c_1b+1)\cdot \left(\dfrac{n-c_1b}{c_1bn}\right)^{2}(\Sigma_{ij}^{2}+\Sigma_{ii}\Sigma_{jj})\nonumber\\
&\ \ \ +((c_1-c_2)b+1)(n-c_1b+1)\cdot \left(\dfrac{n-c_1b}{c_1bn}\right)\left(\dfrac{n-c_2b}{c_2bn}\right)\Sigma_{ij}^{2}. \nonumber
\end{align}
Notice that 
\begin{align}\label{eq:a4express1}
& ((c_1-c_2)b+1)(n-c_1b+1)\cdot \left(\dfrac{n-c_1b}{c_1bn}\right)^{2}\nonumber\\
&=((c_1-c_2)b+1)(n-c_1b+1)\cdot\left(\dfrac{1}{c_1^2b^{2}}+\dfrac{1}{n^{2}}-\dfrac{2}{c_1bn}\right)\nonumber\\
&=(c_1-c_2)b n\dfrac{1}{c_1^2b^{2}} + o\left(\dfrac{n}{b}\right)\nonumber\\
&=(c_1-c_2)\dfrac{n}{c_1^2b}+o\left(\dfrac{n}{b}\right),
\end{align}
and
\begin{align}\label{eq:a4express2}
& ((c_1-c_2)b+1)(n-c_1b+1)\cdot \left(\dfrac{n-c_1b}{c_1bn}\right)\left(\dfrac{n-c_2b}{c_2bn}\right)\nonumber\\
&=((c_1-c_2)b+1)(n-c_1b+1)\cdot\left(\dfrac{1}{c_1c_2b^{2}}+\dfrac{1}{n^{2}}-\dfrac{(c_1+c_1)}{c_1c_2bn}\right)\nonumber\\
&=(c_1-c_2)bn\dfrac{1}{c_1c_2b^{2}}+o\left(\dfrac{n}{b}\right)\nonumber\\
&=\dfrac{c_1-c_2}{c_1c_2}\dfrac{n}{b}+o\left(\dfrac{n}{b}\right).
\end{align}
Plug \eqref{eq:a4express1} and \eqref{eq:a4express2} in \eqref{eq:EOLexpression},
\begin{align}\label{eq:EOLexpressionfinal}
&\ \ \ \ \E \left[((c_1-c_2)b+1)(n-c_1b+1)OL^{(i)}OL^{(j)} \right]\nonumber\\
&=\left[\dfrac{(c_1-c_2)}{c_1^2}(\Sigma_{ii}\Sigma_{jj}+\Sigma_{ij}^{2})+\dfrac{c_1-c_2}{c_1c_2}\Sigma_{ij}^{2}\right]\cdot\dfrac{n}{b}+o\left(\dfrac{n}{b}\right).
\end{align}
We calculate $E[a_4]$ by first deriving
\begin{equation}\label{eq:Cjointdistra4}
\left[ \begin{matrix} \bar{C}_{l}(c_1b)-\bar{C} \\ \bar{C}_{l+(c_1-c_2)b+s}(c_2b)-\bar{C} \end{matrix} \right] \sim N\left( \begin{bmatrix} 0\\ 0\end{bmatrix},\ \left[\begin{matrix}\left(\dfrac{n-c_1b}{c_1bn}\right)\Sigma& \left(\dfrac{1}{c_1b}-\dfrac{1}{n}-\dfrac{s}{c_2^2b^{2}}\right)\Sigma\\ \left(\dfrac{1}{c_1b}-\dfrac{1}{n}-\dfrac{s}{c_2b^{2}}\right)\Sigma& \left(\dfrac{n-c_2b}{c_2bn}\right)\Sigma\end{matrix}\right]\right).
\end{equation}
All we need to obtain is the covariance matrix. Continuing as before in \eqref{eq:Ei!=j},
For $i\neq j$,
\begin{equation}\label{eq:Ei!=ja4}
\E \left[\bar{B}^{(i)}_{l}(c_1b)-\bar{B}^{(i)} \right] \left[\bar{B}^{(j)}_{l+(c_1-c_2)b+s}(c_2b)-\bar{B}^{(j)} \right] =0.
\end{equation}
For $i=j$, we need to calculate $\E[\bar{B}^{(i)}_{l}(c_1b)-\bar{B}^{(i)}][\bar{B}^{(i)}_{l+(c_1-c_2)b+s}(c_1b)-\bar{B}^{(i)}]$ for $s=1,...,(c_2b-1) $. Continuing as before in  \eqref{eq:Ei=j},
\begin{align*}
&\E \left[\bar{B}^{(i)}_{l}(c_1b)-\bar{B}^{(i)} \right]  \left[\bar{B}^{(i)}_{l+(c_1-c_2)b+s}(c_1b)-\bar{B}^{(i)} \right]\\ 
& = \dfrac{1}{c_1b} - \dfrac{1}{n} - \dfrac{s}{c_1c_2b^2}\,. \numberthis \label{eq:Ei=ja4}
\end{align*}

By \eqref{eq:Ei!=ja4} and \eqref{eq:Ei=ja4},
\begin{equation} \label{eq:Ccovdistra4}
\text{Cov}(\bar{C}_{l}(b)-\bar{C},\ \bar{C}_{l+s}(b)-\bar{C})=\left(\dfrac{1}{c_1b} - \dfrac{1}{n} - \dfrac{s}{c_1c_2b^2}\right)\cdot\Sigma. 
\end{equation}
Therefore \eqref{eq:Cjointdistra4} follows from \eqref{eq:Cdistr} and \eqref{eq:Ccovdistra4}. Again denote $Z_{1}=\bar{C}_{l}^{(i)}(c_1b)-\bar{C}^{(i)}$, $Z_{2}=\bar{C}_{l}^{(j)}(c_1b)-\bar{C}^{(j)}$, $Z_{3}=\bar{C}_{l+(c_1-c_2)b+s}^{(i)}(c_2b)-\bar{C}^{(i)}$, $Z_{4}=\bar{C}_{l+(c_1-c_2)b+s}^{(j)}(c_2b)-\bar{C}^{(j)}$, 
\begin{align}\label{eq:Ea4expression}
&E[a_4]\\
&=\sum_{s=1}^{c_2b-1}\sum_{l=0}^{n-c_1b-s}\E \left[(\bar{C}_{l}^{(i)}(c_1b)-\bar{C}^{(i)})(\bar{C}_{l}^{(j)}(c_1b)-\bar{C}^{(j)})(\bar{C}_{l+(c_1-c_2)b+s}^{(i)}(c_2b)-\bar{C}^{(i)})(\bar{C}_{l+(c_1-c_2)b+s}^{(j)}(c_2b)-\bar{C}^{(j)}) \right]\nonumber\\
&=\sum_{s=1}^{c_2b-1}\sum_{l=0}^{n-c_1b-s}E[Z_{1}Z_{2}Z_{3}Z_{4}]\nonumber\\
&=\sum_{s=1}^{c_2b-1}\sum_{l=0}^{n-c_1b-s}\left[\left(\dfrac{n-c_1b}{c_1bn}\right)\left(\dfrac{n-c_2b}{c_2bn}\right)\Sigma_{ij}^{2}+\left(\dfrac{1}{c_1b} - \dfrac{1}{n} - \dfrac{s}{c_1c_2b^2}\right)^{2}\Sigma_{ii}\Sigma_{jj}+\left(\dfrac{1}{c_1b} - \dfrac{1}{n} - \dfrac{s}{c_1c_2b^2}\right)^{2}\Sigma_{ij}^{2}\right]. \nonumber
\end{align}
Notice 
\begin{align}\label{eq:Ea4expression1}
&\sum_{s=1}^{c_2b-1}\sum_{l=0}^{n-c_1b-s}\left(\dfrac{n-c_1b}{c_1bn}\right)\left(\dfrac{n-c_2b}{c_2bn}\right)\nonumber\\
&=\sum_{s=1}^{c_2b-1}\left[\left(\dfrac{1}{c_1c_2b^{2}}-\dfrac{c_2+c_1}{c_1c_2}\dfrac{1}{bn}+\dfrac{1}{n^{2}}\right)(n-c_1b+1)-\left(\dfrac{1}{c_1c_2b^{2}}-\dfrac{c_1+c_2}{c_1c_2}\dfrac{1}{bn}+\dfrac{1}{n^{2}}\right)s\right]\nonumber\\
&=\left(\dfrac{1}{c_1c_2b^{2}}-\dfrac{c_2+c_1}{c_1c_2}\dfrac{1}{bn}+\dfrac{1}{n^{2}}\right)(n-c_1b+1)(c_2b-1)-\left(\dfrac{1}{c_1c_2b^{2}}-\dfrac{c_1+c_2}{c_1c_2}\dfrac{1}{bn}+\dfrac{1}{n^{2}}\right)\left(\dfrac{c^{2}b^{2}}{2}-\dfrac{cb}{2}\right)\nonumber\\
&=\dfrac{1}{c_1c_2b^{2}}\cdot n\cdot c_2b+o\left(\dfrac{n}{b}\right)\nonumber\\
&=\dfrac{n}{c_1b}+o\left(\dfrac{n}{b}\right), 
\end{align}
and
\begin{align}\label{eq:Ea4expression2}
& \sum_{s=1}^{c_2b-1}\sum_{l=0}^{n-c_1b-s}\left(\dfrac{1}{c_1b}-\dfrac{1}{n}-\dfrac{s}{c_1c_2b^{2}}\right)^{2}\nonumber\\
&=\sum_{s=1}^{c_2b-1}\sum_{l=0}^{n-c_1b-s}\left[\dfrac{s^{2}}{c_1^2c_2^{2}b^{4}}+\left(\dfrac{2}{c_1c_2b^{2}n}-\dfrac{2}{c_1^2c_2b^{3}}\right)s+\left(\dfrac{1}{c_1^2b^{2}}+\dfrac{1}{n^{2}}-\dfrac{2}{c_1bn}\right)\right]\nonumber\\
&=\sum_{s=1}^{c_2b-1}\left[-\dfrac{s^{3}}{c_1^2c_2^{2}b^{4}}+\left(\dfrac{n}{c_1^2c_2^{2}b^{4}}-\left(\dfrac{2}{c_1^2c_2}+\dfrac{1}{c_1c_2^{2}}\right)\dfrac{1}{b^{3}}+\dfrac{1}{c_1c_2^{2}}\dfrac{1}{b^{4}} + \dfrac{2}{c_1c_2}\dfrac{1}{b^{2}n}\right)s^{2}\right.\nonumber\\
&\ \ \ \ \ \ \ \ \ \ \ \ \ \ +\left[\left(\dfrac{4}{c_1c_2}-\dfrac{1}{c_1^2}\right)\dfrac{1}{b^{2}}-\dfrac{2}{c_1^2c_2}\dfrac{n}{b^{3}}+\left(\dfrac{2}{c_1}-\dfrac{2}{c_2}\right)\dfrac{1}{bn}+\dfrac{2}{c_1c_2}\dfrac{1}{b^{2}n}-\dfrac{2}{c_1^2c_2}\dfrac{1}{b^{3}}-\dfrac{1}{n^{2}}\right]s\nonumber\\
&\ \ \ \ \ \ \ \ \ \ \ \ \ \ \left.+\left(\dfrac{n}{c_1b^{2}}+\dfrac{3}{n}-\dfrac{3}{c_1b}-\dfrac{c_1b}{n^{2}}+\dfrac{1}{c_1^2b^{2}}+\dfrac{1}{n^{2}}-\dfrac{2}{c_1bn}\right)\right]\nonumber\\
&=-\dfrac{s^{3}}{c_1^2c_2^{2}b^{4}} \left( \dfrac{c_2^2b^4}{4} - \dfrac{c_2^3b^3}{2} + \dfrac{c_2^2b^2}{4}\right)\\ 
& \quad  +\left(\dfrac{n}{c_1^2c_2^{2}b^{4}}-\left(\dfrac{2}{c_1^2c_2}+\dfrac{1}{c_1c_2^{2}}\right)\dfrac{1}{b^{3}}+\dfrac{1}{c_1c_2^{2}}\dfrac{1}{b^{4}} + \dfrac{2}{c_1c_2}\dfrac{1}{b^{2}n}\right)  \cdot  \left(\dfrac{c_2^{3}b^{3}}{3}-\dfrac{c_2^{2}b^{2}}{2}\dfrac{c_2b}{6}\right)\nonumber\\
& \quad +\left[\left(\dfrac{4}{c_1c_2}-\dfrac{1}{c_1^2}\right)\dfrac{1}{b^{2}}-\dfrac{2}{c_1^2c_2}\dfrac{n}{b^{3}}+\left(\dfrac{2}{c_1}-\dfrac{2}{c_2}\right)\dfrac{1}{bn}+\dfrac{2}{c_1c_2}\dfrac{1}{b^{2}n}-\dfrac{2}{c_1^2c_2}\dfrac{1}{b^{3}}-\dfrac{1}{n^{2}}\right] \cdot\left(\dfrac{c_2^{2}b^{2}}{2}-\dfrac{c_2b}{2}\right)\nonumber\\
& \quad \left.+\left(\dfrac{n}{c_1b^{2}}+\dfrac{3}{n}-\dfrac{3}{c_1b}-\dfrac{c_1b}{n^{2}}+\dfrac{1}{c_1^2b^{2}}+\dfrac{1}{n^{2}}-\dfrac{2}{c_1bn}\right) (c_2b - 1)\right]\nonumber\\
%
&=\dfrac{1}{c_1^2c_2^{2}}\cdot\dfrac{n}{b^{4}}\cdot\dfrac{c_2^{3}b^{3}}{3} - \dfrac{2}{c_1^2c_2}\cdot\dfrac{n}{b^{3}}\cdot\dfrac{c_2^{2}b^{2}}{2}+\dfrac{n}{c_1b^{2}}\cdot c_2b+o\left(\dfrac{n}{b}\right)\nonumber\\
&=\dfrac{c_2}{c_1^23}\cdot\dfrac{n}{b}+o\left(\dfrac{n}{b}\right). \nonumber
\end{align}
Plug \eqref{eq:Ea4expression1} and \eqref{eq:Ea4expression2} in \eqref{eq:Ea4expression}
\begin{align}\label{eq:Ea4expressionfinal}
\E[a_4]&=\dfrac{c_2}{3c_1^2}(\Sigma_{ii}\Sigma_{jj}+\Sigma_{ij}^{2})\dfrac{n}{b}+\Sigma_{ij}^{2}\dfrac{n}{c_1b}+o\left(\dfrac{n}{b}\right).
\end{align}
Finally, we calculate $\E[a_5]$. For $s=c_2b,..., (n-c_1b),$ We will show that 
\begin{equation}\label{eq:Cjointdistra5}
\left[ \begin{matrix} \bar{C}_{l}(c_1b)-\bar{C} \\ \bar{C}_{l+(c_1-c_2)b+s}(c_2b)-\bar{C} \end{matrix} \right] \sim N\left( \begin{bmatrix} 0\\ 0\end{bmatrix},\ \left[\begin{matrix}\left(\dfrac{n-c_1b}{c_1bn}\right)\Sigma& -\dfrac{1}{n}\Sigma\\ -\dfrac{1}{n}\Sigma& \left(\dfrac{n-c_2b}{c_2bn}\right)\Sigma\end{matrix}\right]\right).
\end{equation}
For $i\neq j$, 
\begin{equation}\label{eq:Ei!=ja5}
\E \left[\bar{B}^{(i)}_{l}(c_1b)-\bar{B}^{(i)} \right]  \left[\bar{B}^{(j)}_{l+(c_1-c_2)b+s}(c_2b)-\bar{B}^{(j)} \right] = 0.
\end{equation}
For $i=j$, we need to calculate $E[\bar{B}^{(i)}_{l}(b)-\bar{B}^{(i)}][\bar{B}^{(i)}_{l+(1-c)b+s}(cb)-\bar{B}^{(i)}]$ for $s=cb,...,(n-b). $. Similar to the steps in \eqref{eq:Ei=j}, we get
\begin{equation}
\label{eq:Ei=ja5}
\E \left[\bar{B}^{(i)}_{l}(c_1b)-\bar{B}^{(i)} \right]  \left[\bar{B}^{(j)}_{l+(c_1-c_2)b+s}(c_2b)-\bar{B}^{(j)} \right] = -\dfrac{1}{n} \,.
\end{equation}

By \eqref{eq:Ei!=ja5} and \eqref{eq:Ei=ja5},
\begin{equation}\label{eq:Ccovdistra5}
\text{Cov} \left(\bar{C}_{l}(c_1b)-\bar{C},\ \bar{C}_{l+(c_1 - c_2)b + s}(c_2b)-\bar{C} \right)= -\dfrac{1}{n}\cdot\Sigma.
\end{equation}
Therefore \eqref{eq:Cjointdistra5} follows from \eqref{eq:Cdistr} and \eqref{eq:Ccovdistra5}. Again denote $Z_{1}=\bar{C}_{l}^{(i)}(c_1b)-\bar{C}^{(i)}$, $Z_{2}=\bar{C}_{l}^{(j)}(c_1b)-\bar{C}^{(j)}$, $Z_{3}=\bar{C}_{l+(c_1-c_2)b+s}^{(i)}(c_2b)-\bar{C}^{(i)}$, $Z_{4}=\bar{C}_{l+(c_1-c_2)b+s}^{(j)}(c_2b)-\bar{C}^{(j)}$
\begin{align}\label{eq:Ea5expression}
&E[a_5]\\ 
&=\sum_{s=c_2b}^{n-c_1b}\sum_{l=0}^{n-c_1b-s} \E \left[(\bar{C}_{l}^{(i)}(c_1b)-\bar{C}^{(i)})(\bar{C}_{l}^{(j)}(c_1b)-\bar{C}^{(j)})(\bar{C}_{l+(c_1-c_2)b+s}^{(i)}(c_2b)-\bar{C}^{(i)})(\bar{C}_{l+(c_1-c_2)b+s}^{(j)}(c_2b)-\bar{C}^{(j)}) \right]\nonumber\\
&=\sum_{s=c_2b}^{n-c_1b}\sum_{l=0}^{n-c_1b-s} \E[Z_{1}Z_{2}Z_{3}Z_{4}]\nonumber\\
&=\sum_{s=c_2b}^{n-c_1b}\sum_{l=0}^{n-c_1b-s}\left[\left(\dfrac{n-c_1b}{c_1bn}\right)\left(\dfrac{n-c_2b}{c_2bn}\right)\Sigma_{ij}^{2}+\dfrac{1}{n^{2}}\Sigma_{ii}\Sigma_{jj}+\dfrac{1}{n^{2}}\Sigma_{ij}^{2}\right]. \nonumber
\end{align}
Notice
\begin{align}\label{eq:Ea5expression1}
&\sum_{s=c_2b}^{n-c_1b}\sum_{l=0}^{n-c_1b-s}\left(\dfrac{n-c_1b}{c_1bn}\right)\left(\dfrac{n-c_2b}{c_2bn}\right)\nonumber\\
&=\sum_{s=c_2b}^{n-c_1b}\left(\dfrac{1}{c_1c_2b^{2}}-\dfrac{c_1+c_2}{c_1c_2}\dfrac{1}{bn}+\dfrac{1}{n^{2}}\right)(n-c_1b+1) - \left(\dfrac{1}{c_1c_2b^{2}}-\dfrac{c_1+c_2}{c_1c_2}\dfrac{1}{bn}+\dfrac{1}{n^{2}}\right)s\nonumber\\
&=\left(\dfrac{1}{c_1c_2b^{2}}-\dfrac{c_1+c_2}{c_1c_2}\dfrac{1}{bn}+\dfrac{1}{n^{2}}\right)(n-c_1b+1)[n-(c_1+c_2)b+1]\nonumber\\
& \quad -\left(\dfrac{1}{c_1c_2b^{2}}-\dfrac{c_1+c_2}{c_1c_2}\dfrac{1}{bn}+\dfrac{1}{n^{2}}\right)\dfrac{[(c_2-c_1)b+n][n-(c_1+c_2)b+1]}{2}\nonumber\\
&=\left(\dfrac{1}{c_1c_2b^{2}}-\dfrac{c_1+c_2}{c_1c_2}\dfrac{1}{bn}+\dfrac{1}{n^{2}}\right)[n^{2}-(2c_1+c_2)bn+2n+(c_1^2+c_1)b^{2} - (2c_1+c_2)b + 1]\nonumber\\
& \quad \left(\dfrac{1}{c_1c_2b^{2}}-\dfrac{c_1+c_2}{c_1c_2}\dfrac{1}{bn}+\dfrac{1}{n^{2}}\right)\cdot\dfrac{n^{2}-(c_1+c_2)bn+n+(c_2-c_1)bn-(c_1+c_2)(c_2-c_1)b^{2}+(c_2-c_1)b}{2}\nonumber\\
&=\left(\dfrac{1}{c_1c_2b^{2}}\cdot n^{2} -\dfrac{1}{c_1c_2b^{2}}(2c_1+c_2)\cdot bn-\dfrac{c_1+c_2}{c_1c_2}\dfrac{1}{bn}\cdot n^{2}\right)\nonumber\\
&\ \ \ \ \ -\left(\dfrac{1}{c_1c_2b^{2}}\cdot\dfrac{n^{2}}{2}-\dfrac{c_1}{c_1c_2b^{2}}\cdot bn-\dfrac{c_1+c_2}{c_1c_2}\cdot\dfrac{1}{bn}\cdot\dfrac{n^{2}}{2}\right)+o\left(\dfrac{n}{b}\right)\nonumber\\
&=\dfrac{1}{2c_1c_2}\dfrac{n^{2}}{b^{2}}-\left(\dfrac{3}{2c_1}+\dfrac{3}{2c_2}\right)\dfrac{n}{b}+o\left(\dfrac{n}{b}\right),
\end{align}
and
\begin{align}\label{eq:Ea5expression2}
&\sum_{s=c_2b}^{n-c_1b}\sum_{l=0}^{n-c_1b-s}\dfrac{1}{n^{2}}\nonumber\\
&=\sum_{s=c_2b}^{n-c_1b}-\dfrac{s}{n^{2}}+\left(\dfrac{1}{n}-\dfrac{c_1b}{n^{2}}+\dfrac{1}{n^{2}}\right)\nonumber\\
&=-\dfrac{1}{n^{2}}\left( \dfrac{(n - c_1b - c_2b + 1) (n - (c_1 - c_2))}{2}\right)+\left(\dfrac{1}{n}-\dfrac{c_1b}{n^{2}}+\dfrac{1}{n^{2}}\right)\cdot[n-(c_1+c)2b+1]\nonumber\\
&=o\left(\dfrac{n}{b}\right).
\end{align}
Plug \eqref{eq:Ea5expression1} and \eqref{eq:Ea5expression2} in \eqref{eq:Ea5expression},
\begin{align}\label{eq:Ea5expressionfinal}
\E[a_5]&=\Sigma_{ij}^{2} \left( \dfrac{1}{2c_1c_2}\dfrac{n^{2}}{b^{2}}-\left(\dfrac{3}{2c_1}+\dfrac{3}{2c_2}\right)\dfrac{n}{b}\right).
\end{align}
Replace \eqref{eq:EOLexpressionfinal}, \eqref{eq:Ea4expressionfinal} and \eqref{eq:Ea5expressionfinal} in \eqref{eq:A3expression}
\begin{align*}
A_{3}& = \E\left[-\dfrac{c_1c_2b^{2}}{n^{2}}\left[\sum_{l=0}^{n-b}(\bar{C}_{l}^{(i)}(b)-\bar{C}^{(i)})(\bar{C}_{l}^{(j)}(b)-\bar{C}^{(j)})\right]\left[\sum_{l=0}^{n-cb}(\bar{C}_{l}^{(i)}(cb)-\bar{C}^{(i)})(\bar{C}_{l}^{(j)}(cb)-\bar{C}^{(j)})\right]\right]\notag\\
&=- \dfrac{c_1c_2b^{2}}{n^{2}}\cdot \E \left[((c_1-c_2)b+1)(n-b+1)\cdot OL^{(i)}OL^{(j)} +2a_4 + 2a_5 \right] \\ 
&=-c_1c_2\dfrac{b}{n}\cdot \left[\dfrac{(c_1-c_2)}{c_1^2}(\Sigma_{ii}\Sigma_{jj}+\Sigma_{ij}^{2})+\dfrac{c_1-c_2}{c_1c_2}\Sigma_{ij}^{2}\right.\notag\\
&\quad \quad \quad  + \dfrac{2c_2}{3c_1^2}(\Sigma_{ii}\Sigma_{jj}+\Sigma_{ij}^{2})+\Sigma_{ij}^{2}\dfrac{2}{c_1} + \left. \Sigma_{ij}^{2} \left( \dfrac{1}{c_1c_2}\dfrac{n}{b}-\left(\dfrac{3}{c_1}+\dfrac{3}{c_2}\right)\right) \right]+o\left(\dfrac{b}{n}\right)\notag\\
&=-c_1c_2 \dfrac{b}{n} \left[\dfrac{3c_1 - c_2}{3c_1^2}(\Sigma_{ii}\Sigma_{jj}+\Sigma_{ij}^{2})\cdot - 2\left(\dfrac{c_1+c_2}{c_1c_2}\right)\cdot\Sigma_{ij}^{2} +\dfrac{1}{c_1c_2} \dfrac{n}{b}\cdot\Sigma_{ij}^{2} \right]+o\left(\dfrac{b}{n}\right) \notag\\
&= \dfrac{(c_2 - 3c_1)c_2}{3c_1}(\Sigma_{ii}\Sigma_{jj}  + \Sigma_{ij}^{2})\cdot \dfrac{b}{n} + 2\left(c_1+c_2\right)\cdot\Sigma_{ij}^{2}\cdot\dfrac{b}{n} - \Sigma_{ij}^{2} +o\left(\dfrac{b}{n}\right)\,.
\end{align*}
\end{proof}
Define
$$\tilde{\Sigma}_{wL}=\dfrac{1}{n}\sum_{k=1}^{b}\sum_{l=0}^{n-k}k^{2}\Delta_{2}w_{n}(k)[\bar{C}_{l}(k)-\bar{C}][\bar{C}_{l}(k)-\bar{C}]^T\,,
$$
with elements $\tilde{\Sigma}_{wL, ij}.$
\begin{lemma}\label{lemma:varw}
If Assumption~\ref{ass:batch} holds and $\sum_{k=1}^{b}(\Delta_{2}w_{k})^2\leq O\left(b^{-2}\right)$
then
\begin{equation}\label{eq:varwLfinal}
Var[\tilde{\Sigma}_{wL,ij}]=(\Sigma_{ii}\Sigma_{jj}+\Sigma_{ij}^{2})\bigg[\dfrac{2}{3}\sum_{k=1}^{b}(\Delta_{2}w_{k})^{2}k^{3}\cdot\dfrac{1}{n}+2\sum_{t=1}^{b-1}\sum_{u=1}^{b-t}\Delta_{2}w_{u}\Delta_{2}w_{t+u}\left(\dfrac{2}{3}u^{3}+u^{2}t\right)\cdot\dfrac{1}{n}\bigg]+o\left(\dfrac{b}{n}\right).
\end{equation}
\end{lemma}
\begin{proof}
Note 
$(\Delta_{2}w_{k})^2\leq \sum_{k=1}^{b}(\Delta_{2}w_{k})^2 = O\left(b^{-2}\right),$
hence
$a_{k}=b\cdot \Delta_{2}w_{k} =  O(1).$
Consider
$$\tilde{\Sigma}_{wL,ij}=\dfrac{1}{n}\sum_{k=1}^{b}\sum_{l=0}^{n-k}k^{2}\Delta_{2}w_{n}(k)[\bar{C}_{l}^{(i)}(k)-\bar{C}^{(i)}][\bar{C}^{(j)}_{l}(k)-\bar{C}^{(j)}].$$
Let $c_{k}=k/b$ for $k=1,...,b$, also denote $a_{k}=b\cdot\Delta_{2}w_{k}$ for simplicity. 
Hence
\begin{align*}
\tilde{\Sigma}_{wL,ij}&=\dfrac{1}{n}\sum_{k=1}^{b}\sum_{l=0}^{n-k}k^{2}\Delta_{2}w_{n}(k)[\bar{C}_{l}^{(i)}(k)-\bar{C}^{(i)}][\bar{C}^{(j)}_{l}(k)-\bar{C}^{(j)}]\nonumber\\
&=\dfrac{1}{n}\sum_{k=1}^{b}\sum_{l=0}^{n-k}c_{k}^{2}b^{2}\Delta_{2}w_{n}(k)[\bar{C}_{l}^{(i)}(k)-\bar{C}^{(i)}][\bar{C}^{(j)}_{l}(k)-\bar{C}^{(j)}]\nonumber\\
&=\sum_{k=1}^{b}c_{k}a_{k}\left(\dfrac{c_{k}b}{n}\sum_{l=0}^{n-c_{k}b}[\bar{C}_{l}^{(i)}(k)-\bar{C}^{(i)}][\bar{C}^{(j)}_{l}(k)-\bar{C}^{(j)}]\right).
\end{align*}


Define $A_{1,ij}^{(k)}$ and $A_{2,ij}^{(ut)}$ below and apply Lemma~\ref{lemma:A2_var_term},
\begin{align}\label{eq:A1kwL}
A_{1,ij}^{(k)} & = \E\left[\dfrac{(c_{k}b)^{2}}{n^{2}}\cdot\left(\sum_{k=0}^{n-c_{k}b}(\bar{C}_{l}^{(i)}(c_{k}b)-\bar{C}^{(i)})(\bar{C}_{l}^{(j)}(c_{k}b)-\bar{C}^{(j)})\right)^{2}\right]\nonumber\\
&=\left(\dfrac{2}{3}(\Sigma_{ii}\Sigma_{jj}+\Sigma_{ij}^{2})-4\Sigma_{ij}^{2}\right)\cdot\dfrac{c_{k}b}{n}+\Sigma_{ij}^{2}+o\left(\dfrac{b}{n}\right)\,,
\end{align}
and
\begin{align}\label{eq:A1utwL}
A_{2,ij}^{(ut)}& = \E\left[\dfrac{(c_{u+t}b)^{2}}{n^{2}}\cdot\left(\sum_{p=0}^{n-c_{t}b}(\bar{C}_{p}^{(i)}(c_{t}b)-\bar{C}^{(i)})(\bar{C}_{p}^{(j)}(c_{t}b)-\bar{C}^{(j)})\right)\right.\nonumber\\
&\ \ \ \ \ \ \ \ \ \ \ \ \ \ \ \ \ \ \ \ \ \ \ \ \ \left.\cdot\left(\sum_{q=0}^{n-c_{u+t}b}(\bar{C}^{(i)}_{q}(c_{t+u}b)-\bar{C}^{(i)})(\bar{C}^{(j)}_{q}(c_{t+u}b)-\bar{C}^{(j)})\right)\right]\nonumber\\
&=\left[\left(c_{u+t}-\dfrac{c_{u}}{3}\right)(\Sigma_{ii}\Sigma_{jj}+\Sigma_{ij}^{2})-\left(2c_{u+t}+\dfrac{2c_{u+t}^{2}}{c_{u}}\right)\Sigma_{ij}^{2}\right]\dfrac{b}{n}+\dfrac{c_{u+t}}{c_{u}}\Sigma_{ij}^{2}+o\left(\dfrac{b}{n}\right).
\end{align}
To calculate $\Var[\tilde{\Sigma}_{wL,ij}]$, we will calculate $\E[\tilde{\Sigma}_{wL,ij}^{2}]$ and $(\E[\tilde{\Sigma}_{wL,ij}])^{2}$. Plugging \eqref{eq:A1kwL} and \eqref{eq:A1utwL} in the expression of $\E[\tilde{\Sigma}_{wL,ij}^{2}]$ results in
\begin{align}\label{eq:Ex^2wLfinal}
& \E[\tilde{\Sigma}_{wL,ij}^{2}]\\ 
& = \E\left[\left(\sum_{k=1}^{b}c_{k}a_{k}\cdot\left[\dfrac{c_{k}b}{n}\sum_{l=0}^{n-c_{k}b}(\bar{C}_{l}^{(i)}(k)-\bar{C}^{(i)})(\bar{C}_{l}^{(j)}(k)-\bar{C}^{(j)})\right]\right)^{2}\right]\nonumber\\
& = \E\left[\sum_{k=1}^{b}\left(c_{k}a_{k}\cdot\left[\dfrac{c_{k}b}{n}\sum_{l=0}^{n-c_{k}b}(\bar{C}_{l}^{(i)}(k)-\bar{C}^{(i)})(\bar{C}_{l}^{(j)}(k)-\bar{C}^{(j)})\right]\right)^{2}\right.\nonumber\\
&\ \ \ \ +2\sum_{t=1}^{b-1}\sum_{u=1}^{b-t}c_{u}^{2}a_{u}c_{t+u}^{2}a_{t+u}\cdot\dfrac{b^{2}}{n^{2}}\left(\sum_{p=0}^{n-c_{t}b}(\bar{C}_{p}^{(i)}(c_{t}b)-\bar{C}^{(i)})(\bar{C}_{p}^{(j)}(c_{t}b)-\bar{C}^{(j)})\right)\nonumber\\
&\ \ \ \ \ \ \ \ \ \ \ \ \ \ \ \ \ \ \ \ \ \ \ \ \ \ \ \ \ \ \ \ \ \ \left.\cdot\left(\sum_{q=0}^{n-c_{t+u}b}(\bar{C}_{q}^{(i)}(c_{t+u}b)-\bar{C}^{(i)})(\bar{C}_{q}^{(j)}(c_{t+u}b)-\bar{C}^{(j)})\right)\right]\nonumber\\
&=\sum_{k=1}^{b}c_{k}^{2}a_{k}^{2}\cdot E\left[\dfrac{(c_{k}b)^{2}}{n^{2}}\cdot\left(\sum_{l=0}^{n-c_{k}b}(\bar{C}_{l}^{(i)}(k)-\bar{C}^{(i)})(\bar{C}_{l}^{(j)}(k)-\bar{C}^{(j)})\right)^{2}\right]\nonumber\\
&\ \ \ \ +2\sum_{t=1}^{b-1}\sum_{u=1}^{b-t}c_{u}^{2}a_{u}a_{t+u}\cdot E\left[\dfrac{(c_{u+t}b)^{2}}{n^{2}}\cdot\left(\sum_{p=0}^{n-c_{t}b}(\bar{C}_{p}^{(i)}(c_{t}b)-\bar{C}^{(i)})(\bar{C}_{p}^{(j)}(c_{t}b)-\bar{C}^{(j)})\right)\right.\nonumber\\
&\ \ \ \ \ \ \ \ \ \ \ \ \ \ \ \ \ \ \ \ \ \ \ \ \ \ \ \ \ \ \ \ \ \  \cdot\left.\left(\sum_{q=0}^{n-c_{t+u}b}(\bar{C}_{q}^{(i)}(c_{t+u}b)-\bar{C}^{(i)})(\bar{C}_{q}^{(j)}(c_{t+u}b)-\bar{C}^{(j)})\right)\right]\nonumber\\
&=\sum_{k=1}^{b}c_{k}^{2}a_{k}^{2}A_{1,ij}^{(k)}+2\sum_{t=1}^{b-1}\sum_{u=1}^{b-t}c_{u}^{2}a_{u}a_{u+t}A_{2,ij}^{(ut)}\nonumber\\
&=o\left(\dfrac{b}{n}\right)+\sum_{k=1}^{b}c_{k}^{2}a_{k}^{2}\left[\left(\dfrac{2}{3}(\Sigma_{ii}\Sigma_{jj}+\Sigma_{ij}^{2})-4\Sigma_{ij}^{2}\right)\cdot\dfrac{c_{k}b}{n}+\Sigma_{ij}^{2}\right]\nonumber\\
&+2\sum_{t=1}^{b-1}\sum_{u=1}^{b-t}c_{u}^{2}a_{u}a_{u+t}\left[\left[\left(c_{u+t}-\dfrac{c_{u}}{3}\right)(\Sigma_{ii}\Sigma_{jj}+\Sigma_{ij}^{2})-\left(2c_{u+t}+\dfrac{2c_{u+t}^{2}}{c_{u}}\right)\Sigma_{ij}^{2}\right]\cdot\dfrac{b}{n}+\dfrac{c_{u+t}}{c_{u}}\Sigma_{ij}^{2}\right]\nonumber\\
&=\left[\sum_{k=1}^{b}c_{k}^{2}a_{k}^{2}\cdot\Sigma_{ij}^{2}+2\sum_{t=1}^{b-1}\sum_{u=1}^{b-t}c_{u}^{2}a_{u}a_{u+t}\dfrac{c_{u+t}}{c_{u}}\cdot\Sigma_{ij}^{2}\right]+\left[\sum_{k=1}^{b}c_{k}^{3}a_{k}^{2}\left(\dfrac{2}{3}[\Sigma_{ii}\Sigma_{jj}+\Sigma_{ij}^{2}]-4\Sigma_{ij}^{2}\right)\cdot\dfrac{b}{n}\right.\nonumber\\
&\ \ \ \ \ +\left.2\sum_{t=1}^{b-1}\sum_{u=1}^{b-t}c_{u}^{2}a_{u}a_{u+t}\left[\left(c_{u+t}-\dfrac{c_{u}}{3}\right)(\Sigma_{ii}\Sigma_{jj}+\Sigma_{ij}^{2})-\left(2c_{u+t}+\dfrac{2c_{u+t}^{2}}{c_{u}}\right)\Sigma_{ij}^{2}\right]\dfrac{b}{n}\right]+o\left(\dfrac{b}{n}\right)\nonumber\\
&=\left( \sum_{k=1}^{b}a_{k}c_{k} \right)^{2}\Sigma_{ij}^{2}+\sum_{k=1}^{b}c_{k}^{3}a_{k}^{2}\left(\dfrac{2}{3}(\Sigma_{ii}\Sigma_{jj}+\Sigma_{ij}^{2})-4\Sigma_{ij}^{2}\right)\cdot\dfrac{b}{n}\nonumber\\
&\ \ \ \ \ +2\sum_{t=1}^{b-1}\sum_{u=1}^{b-t}c_{u}^{2}a_{u}a_{u+t}\left[\left(c_{u+t}-\dfrac{c_{u}}{3}\right)(\Sigma_{ii}\Sigma_{jj}+\Sigma_{ij}^{2})-\left(2c_{u+t}+\dfrac{2c_{u+t}^{2}}{c_{u}}\right)\Sigma_{ij}^{2}\right]\cdot\dfrac{b}{n}+o\left(\dfrac{b}{n}\right).
\end{align}
By \eqref{eq:Cdistr},
\begin{equation}\label{eq:(Ex)^2wL}
\E[(C_{l}^{(i)}(c_{k}b)-\bar{C}^{(i)})(C_{l}^{(j)}(c_{k}b)-\bar{C}^{(j)})]=\dfrac{n-c_{k}b}{c_{k}bn}\Sigma_{ij}.
\end{equation}
Plug \eqref{eq:(Ex)^2wL} in $(\E[\tilde{\Sigma}_{wL,ij}])^{2}$,
\begin{align}\label{eq:(Ex)^2wLfinal}
& (\E[\tilde{\Sigma}_{wL,ij}])^{2}) \\ 
&=\left(\dfrac{1}{n}\sum_{k=1}^{b}\sum_{l=0}^{n-k}k^{2}\Delta_{2}w_{k} \E \left[(C_{l}^{(i)}(c_{k}b)-\bar{C}^{(i)})(C_{l}^{(j)}(c_{k}b)-\bar{C}^{(j)}) \right]\right)^{2}\nonumber\\
&=\left(\sum_{k=1}^{b}c_{k}a_{k}\left[\dfrac{c_{k}b}{n}\sum_{l=0}^{n-c_{k}b} \E \left[(C_{l}^{(i)}(c_{k}b)-\bar{C}^{(i)})(C_{l}^{(j)}(c_{k}b)-\bar{C}^{(j)}) \right]\right]\right)^{2}\nonumber\\
&=\left(\sum_{k=1}^{b}c_{k}a_{k}\left[\dfrac{c_{k}b}{n}\cdot(n-c_{k}b+1)\cdot\dfrac{n-c_{k}b}{c_{k}bn}\cdot\Sigma_{ij}\right]\right)^{2}\ \text{apply\ (1.4.3)}\nonumber\\
&=\Sigma_{ij}^{2}\left[(\sum_{k=1}^{b}a_{k}c_{k})^{2}-\sum_{k=1}^{b}4a_{k}^{2}c_{k}^{3}\cdot\dfrac{b}{n}-2\sum_{t=1}^{b-1}\sum_{u=1}^{b-t}a_{u}a_{u+t}(2c_{u}^{2}c_{u+t}+2c_{u}c_{u+t}^{2})\cdot\dfrac{b}{n}\right]+o\left(\dfrac{b}{n}\right). \nonumber
\end{align}
Combine \eqref{eq:Ex^2wLfinal} and \eqref{eq:(Ex)^2wLfinal},
\begin{align*}
& \Var[\tilde{\Sigma}_{wL,ij}] = \E[\tilde{\Sigma}_{wL,ij}^{2}] - (\E[\tilde{\Sigma}_{wL,ij}])^{2}\nonumber\\
&=\sum_{k=1}^{b}c_{k}^{3}a_{k}^{2}\Bigg(\left[\dfrac{2}{3}(\Sigma_{ii}\Sigma_{jj}+\Sigma_{ij}^{2})-4\Sigma_{ij}^{2}\right]+4\Sigma_{ij}^{2}\Bigg)\cdot\dfrac{b}{n}\nonumber\\
&\ \ \ \ \ +2\sum_{t=1}^{b-1}\sum_{u=1}^{b-t}\Bigg(c_{u}^{2}a_{u}a_{u+t}\left[\left(c_{u+t}-\dfrac{c_{u}}{3}\right)(\Sigma_{ii}\Sigma_{jj}+\Sigma_{ij}^{2})-\left(2c_{u+t}+\dfrac{2c_{u+t}^{2}}{c_{u}}\right)\Sigma_{ij}^{2}\right]\nonumber\\
&\ \ \ \ \ \ \ \ \ \ +a_{u}a_{u+t}(2c_{u}^{2}c_{u+t}+2c_{u}c_{u+t}^{2})\Sigma_{ij}^{2}\Bigg)\cdot\dfrac{b}{n}+o\left(\dfrac{b}{n}\right)\nonumber\\
&=\sum_{k=1}^{b}\dfrac{2}{3} c_{k}^{3}a_{k}^{2}(\Sigma_{ii}\Sigma_{jj}+\Sigma_{ij}^{2})\cdot\dfrac{b}{n}+2\sum_{t=1}^{b}\sum_{u=1}^{b-t}\left(c_{u}^{2}c_{u+t}-\dfrac{1}{3}c_{u}^{3}\right)a_{u}a_{u+t}(\Sigma_{ii}\Sigma_{jj}+\Sigma_{ij}^{2})\cdot\dfrac{b}{n}+o\left(\dfrac{b}{n}\right)\nonumber\\
&=\sum_{k=1}^{b}\dfrac{2}{3}\left(\dfrac{k}{b}\right)^{3}(b\Delta_{2}w_{k})^{2}(\Sigma_{ii}\Sigma_{jj}+\Sigma_{ij}^{2})\cdot\dfrac{b}{n}\nonumber\\
&\ \ \ \ +2\sum_{t=1}^{b-1}\sum_{u=1}^{b-t}\left[\left(\left(\dfrac{u}{b}\right)^{2}\dfrac{u+t}{b}-\dfrac{1}{3}\left(\dfrac{u}{b}\right)^{3}\right)b\Delta_{2}w_{u}\cdot b\Delta_{2}w_{u+t}\right](\Sigma_{ii}\Sigma_{jj}+\Sigma_{ij}^{2})\cdot\dfrac{b}{n}+o\left(\dfrac{b}{n}\right)\nonumber\\
&=(\Sigma_{ii}\Sigma_{jj}+\Sigma_{ij}^{2})\cdot\bigg[\dfrac{2}{3}\sum_{k=1}^{b}(\Delta_{2}w_{k})^{2}k^{3}\cdot\dfrac{1}{n}+2\sum_{t=1}^{b-1}\sum_{u=1}^{b-t}\Delta_{2}w_{u}\cdot\Delta_{2}w_{t+u}\cdot\left(\dfrac{2}{3}u^{3}+u^{2}t\right)\cdot\dfrac{1}{n}\bigg]+o\left(\dfrac{b}{n}\right).
\end{align*}
\end{proof}

\begin{lemma}\label{lemma:hat}
\citep[Lemma 14][]{vats:fleg:jone:2015spec} Suppose \eqref{eq:sip} holds for $f = g$ and Assumption~\ref{ass:batch} hold. If, as $n\to \infty$,
\[
b\psi(n)^{2} \log n\left(\sum_{k=1}^{b}|\Delta_{2}w_{n}(k)|\right)^{2}\rightarrow 0,
\]
and
\[
\psi(n)^{2}\sum_{k=1}^{b}|\Delta_{2}w_{n}(k)|\rightarrow 0,
\]
then $\hat{\Sigma}_{w}\rightarrow \tilde{\Sigma}_{wL}$ as $n \to \infty$ w.p.\ 1.
\end{lemma}

\begin{lemma}\label{lemma:Esq0}
Suppose \eqref{eq:sip} holds for $f = g$ and $f = g^2$ (where the square is element-wise) such that $\E_FD^4 < \infty$ and Assumption~\ref{ass:batch} holds. Further, suppose $\psi^2(n)b^{-1}\log n \rightarrow 0$, then 
\begin{equation*}
\E[\hat{\Sigma}_{w,ij}-\tilde{\Sigma}_{wL,ij}]^{2}\rightarrow\ 0\  \text{as}\ n\rightarrow\infty.
\end{equation*}
\end{lemma}

\begin{proof}
An observation of Lemma B.4 of \cite{jone:hara:caff:neat:2006}, Lemmas 12, 13 and 14 of \cite{fleg:jone:2010} and Lemma 5 of \cite{liu:fleg:2018} show that Lemma~\ref{lemma:Esq0} hold.
\end{proof}

\section{Proof of Theorem~\ref{thm:var}}\label{proof theor var}
Define 
$$\eta= \Var[\hat{\Sigma}_{w,ij}-\tilde{\Sigma}_{wL,ij}]+2\E[(\hat{\Sigma}_{w,ij}-\tilde{\Sigma}_{wL,ij})(\tilde{\Sigma}_{wL,ij} - \E\tilde{\Sigma}_{wL,ij})],$$
we first show that $\eta\rightarrow 0$ as $n\rightarrow\infty$. Apply Lemma~\ref{lemma:Esq0}, by Cauchy-Schwarz inequality and $\Var[X]\leq \E X^{2}$.
\begin{align*}
|\eta|&=|\Var[\hat{\Sigma}_{w,ij}-\tilde{\Sigma}_{wL,ij}] + 2\E[(\hat{\Sigma}_{w,ij}-\tilde{\Sigma}_{wL,ij})(\tilde{\Sigma}_{wL,ij} - \E\tilde{\Sigma}_{wL,ij})]|\nonumber\\
&\leq\E[\hat{\Sigma}_{w,ij}-\tilde{\Sigma}_{wL,ij}]^{2}+2\sqrt{\E[\hat{\Sigma}_{w,ij}-\tilde{\Sigma}_{wL,ij}]^{2}\cdot\E[\tilde{\Sigma}_{w,L,ij} - \E\tilde{\Sigma}_{wL, ij}]^{2}}\nonumber\\
&=E[\hat{\Sigma}_{w,ij}-\tilde{\Sigma}_{w,L,ij}]^{2}+2(E[\hat{\Sigma}_{w,ij}-\tilde{\Sigma}_{wL,ij}]^{2})^{1/2}\cdot (Var[\tilde{\Sigma}_{wL,ij}])^{1/2}
\end{align*}
By the conditions of Lemma~\ref{lemma:varw}, 
\[\dfrac{1}{n}\sum_{k=1}^{b}(\Delta_2w_k)^2k^3\leq\dfrac{b^3}{n}\sum_{k=1}^{b}(\Delta_2w_k)^2\leq O \left(\dfrac{b}{n} \right).\]
Hence \eqref{eq:varwLfinal} can be written as
\[ 
\Var[\tilde{\Sigma}_{wL,ij}]=((\Sigma_{ii}\Sigma_{jj} + \Sigma_{ij}^2)S \dfrac{b}{n}+o(1))\cdot\dfrac{b}{n}.
\]
By Lemma~\ref{lemma:Esq0}, $\E[\hat{\Sigma}_{w,ij}-\tilde{\Sigma}_{wL,ij}]^{2}=o(1)$, therefore
\begin{align}\label{eq:eta}
|\eta|&\leq\E[\hat{\Sigma}_{w,ij}-\tilde{\Sigma}_{,L,ij}]^{2}+2(\E[\hat{\Sigma}_{w,ij}-\tilde{\Sigma}_{wL,ij}]^{2})^{1/2}\cdot (\Var[\tilde{\Sigma}_{wL,ij}])^{1/2}\nonumber\\
&=o(1)+2\sqrt{o(1)\cdot[((\Sigma_{ii}\Sigma_{jj} + \Sigma_{ij}^2)S+o(1))\cdot\dfrac{b}{n}]}\nonumber\\
&=o(1)+2\left(\dfrac{b}{n}\right)^{1/2}[o(1)\cdot((\Sigma_{ii}\Sigma_{jj} + \Sigma_{ij}^2)S+o(1))]^{1/2} =o(1).
\end{align}
Since $b/n\rightarrow 0$ as $n\rightarrow\infty$, plug in \eqref{eq:eta}
\begin{align*}
\Var[\hat{\Sigma}_{w,ij}]& = \E[\hat{\Sigma}_{w,ij} - \E\hat{\Sigma}_{w,ij}]^{2}\\
& = \E[\hat{\Sigma}_{w,ij}-\tilde{\Sigma}_{wL,ij}+\tilde{\Sigma}_{wL,ij} - \E\tilde{\Sigma}_{wL,ij} + \E\tilde{\Sigma}_{wL,ij} - \E\hat{\Sigma}_{w,ij}]^{2}\\
& = \E[(\hat{\Sigma}_{w,ij}-\tilde{\Sigma}_{wL,ij})+(\tilde{\Sigma}_{wL,ij} - \E\tilde{\Sigma}_{wL,ij}) - (\E\hat{\Sigma}_{w,ij} - \E\tilde{\Sigma}_{wL,ij})]^{2}\\
& = \E[(\hat{\Sigma}_{w,ij}-\tilde{\Sigma}_{wL,ij}) - \E(\hat{\Sigma}_{w,ij}-\tilde{\Sigma}_{wL,ij})]^{2} + \E[\tilde{\Sigma}_{wL,ij} - \E\tilde{\Sigma}_{wL,ij}]^{2}\\
&\ \ \ \ \ + 2 \E[[(\hat{\Sigma}_{w,ij}-\tilde{\Sigma}_{wL,ij}) - \E(\hat{\Sigma}_{w,ij}-\tilde{\Sigma}_{wL,ij})]\cdot [\tilde{\Sigma}_{w,ij} - \E\tilde{\Sigma}_{w,ij}]]\\
& = \E[(\hat{\Sigma}_{w,ij}-\tilde{\Sigma}_{w,L,ij}) - \E(\hat{\Sigma}_{w,ij}-\tilde{\Sigma}_{wL,ij})]^{2} + \E[\tilde{\Sigma}_{wL,ij} - \E\tilde{\Sigma}_{wL,ij}]^{2}\\
&\ \ \ \ \  + 2 \E[(\hat{\Sigma}_{w,ij}-\tilde{\Sigma}_{wL,ij})\cdot (\tilde{\Sigma}_{w,ij} - \E\tilde{\Sigma}_{w,ij})]\\
& = \E[\tilde{\Sigma}_{wL,ij} - \E\tilde{\Sigma}_{wL,ij}]^{2}+\eta\\
&=(\Sigma_{ii}\Sigma_{jj} + \Sigma_{ij}^2)S\cdot\dfrac{b}{n}+o\left(\dfrac{b}{n}\right)+o(1).
\end{align*}

\end{appendix}
\setstretch{1}
\bibliographystyle{apalike}
\bibliography{ref}

\begin{thebibliography}{}

\bibitem[Andrews, 1991]{andr:1991}
Andrews, D. (1991).
\newblock Heteroskedasticity and autocorrelation consistent covariant matrix
  estimation.
\newblock {\em Econometrica}, 59:817--858.

\bibitem[Atchad{\'e}, 2011]{atch:2011}
Atchad{\'e}, Y.~F. (2011).
\newblock Kernel estimators of asymptotic variance for adaptive {M}arkov chain
  {M}onte {C}arlo.
\newblock {\em The Annals of Statistics}, 39(2):990--1011.

\bibitem[Boone et~al., 2014]{boon:merr:krac:2014}
Boone, E.~L., Merrick, J.~R., and Krachey, M.~J. (2014).
\newblock A {H}ellinger distance approach to {MCMC} diagnostics.
\newblock {\em Journal of Statistical Computation and Simulation},
  84(4):833--849.

\bibitem[Brockmann et~al., 1993]{broc:mich:gass:theo:herr:1993}
Brockmann, M., Gasser, T., and Herrmann, E. (1993).
\newblock Locally adaptive bandwidth choice for kernel regression estimators.
\newblock {\em Journal of the American Statistical Association},
  88(424):1302--1309.

\bibitem[B{\"u}hlmann, 1996]{buhlmann1996locally}
B{\"u}hlmann, P. (1996).
\newblock Locally adaptive lag-window spectral estimation.
\newblock {\em Journal of Time Series Analysis}, 17(3):247--270.

\bibitem[Chan and Yau, 2017]{chan2017automatic}
Chan, K.~W. and Yau, C.~Y. (2017).
\newblock Automatic optimal batch size selection for recursive estimators of
  time-average covariance matrix.
\newblock {\em Journal of the American Statistical Association},
  112(519):1076--1089.

\bibitem[Damerdji, 1991]{dame:1991}
Damerdji, H. (1991).
\newblock Strong consistency and other properties of the spectral variance
  estimator.
\newblock {\em Management Science}, 37:1424--1440.

\bibitem[Damerdji, 1995]{dame:1995}
Damerdji, H. (1995).
\newblock Mean-square consistency of the variance estimator in steady-state
  simulation output analysis.
\newblock {\em Operations Research}, 43:282--291.

\bibitem[Elith et~al., 2008]{elit:leat:2008}
Elith, J., Leathwick, J., and Hastie, T. (2008).
\newblock A working guide to boosted regression trees.
\newblock {\em Journal of Animal Ecology}, 77(4):802--813.

\bibitem[Finley and Banerjee, 2013]{finl:bana:2013}
Finley, A.~O. and Banerjee, S. (2013).
\newblock sp{B}ayes: Univariate and multivariate spatial modeling {R} package
  version 0.3-7.
\newblock http://CRAN.R-project.org/package=spBayes.

\bibitem[Finley et~al., 2012]{finl:bane:gelf:2012}
Finley, A.~O., Banerjee, S., and Gelfand, A.~E. (2012).
\newblock Bayesian dynamic modeling for large space-time datasets using
  gaussian predictive processes.
\newblock {\em Journal of Geographical Systems}, 14(1):29--47.

\bibitem[Flegal et~al., 2008]{fleg:hara:jone:2008}
Flegal, J.~M., Haran, M., and Jones, G.~L. (2008).
\newblock {M}arkov chain {M}onte {C}arlo: Can we trust the third significant
  figure?
\newblock {\em Statistical Science}, 23:250--260.

\bibitem[Flegal et~al., 2017]{mcmcse:2017}
Flegal, J.~M., Hughes, J., Vats, D., and Dai, N. (2017).
\newblock mcmcse: Monte {C}arlo standard errors for {MCMC} {R} package version
  1.3-2.
\newblock http://cran.r-project.org/web/packages/mcmcse/index.html.

\bibitem[Flegal and Jones, 2010]{fleg:jone:2010}
Flegal, J.~M. and Jones, G.~L. (2010).
\newblock Batch means and spectral variance estimators in {M}arkov chain
  {M}onte {C}arlo.
\newblock {\em The Annals of Statistics}, 38:1034--1070.

\bibitem[Geyer, 2011]{geye:2011}
Geyer, C.~J. (2011).
\newblock Introduction to {M}arkov chain {M}onte {C}arlo.
\newblock In {\em Handbook of Markov Chain Monte Carlo}. CRC, London.

\bibitem[Glynn and Whitt, 1992]{glyn:whit:1992}
Glynn, P.~W. and Whitt, W. (1992).
\newblock The asymptotic validity of sequential stopping rules for stochastic
  simulations.
\newblock {\em The Annals of Applied Probability}, 2:180--198.

\bibitem[Janssen and Stoica, 1987]{jans:pete:stoi:petr:1987}
Janssen, P.~H. and Stoica, P. (1987).
\newblock {\em On the expectation of the product of four matrix-valued Gaussian
  random variables}.
\newblock Eindhoven University of Technology.

\bibitem[Jones, 2004]{jone:2004}
Jones, G.~L. (2004).
\newblock On the {M}arkov chain central limit theorem.
\newblock {\em Probability Surveys}, 1:299--320.

\bibitem[Jones et~al., 2006]{jone:hara:caff:neat:2006}
Jones, G.~L., Haran, M., Caffo, B.~S., and Neath, R. (2006).
\newblock Fixed-width output analysis for {M}arkov chain {M}onte {C}arlo.
\newblock {\em Journal of the American Statistical Association},
  101:1537--1547.

\bibitem[Jones and Hobert, 2001]{jone:hobe:2001}
Jones, G.~L. and Hobert, J.~P. (2001).
\newblock Honest exploration of intractable probability distributions via
  {M}arkov chain {M}onte {C}arlo.
\newblock {\em Statistical Science}, 16:312--334.

\bibitem[Jones et~al., 1996]{Jone:Chri:Jame:Shea:Simon:1996}
Jones, M.~C., Marron, J.~S., and Sheather, S.~J. (1996).
\newblock A brief survey of bandwidth selection for density estimation.
\newblock {\em Journal of the American Statistical Association},
  91(433):401--407.

\bibitem[Kuelbs and Philipp, 1980]{kuelbs:phil:1980}
Kuelbs, J. and Philipp, W. (1980).
\newblock Almost sure invariance principles for partial sums of mixing $ b
  $-valued random variables.
\newblock {\em The Annals of Probability}, 8:1003--1036.

\bibitem[Leathwick et~al., 2008]{leat:elit:2008}
Leathwick, J., Elith, J., Chadderton, W., Rowe, D., and Hastie, T. (2008).
\newblock Dispersal, disturbance and the contrasting biogeographies of {N}ew
  {Z}ealand’s diadromous and non-diadromous fish species.
\newblock {\em Journal of Biogeography}, 35(8):1481--1497.

\bibitem[Liu and Flegal, 2018]{liu:fleg:2018}
Liu, Y. and Flegal, J. (2018).
\newblock Weighted batch means estimators in {M}arkov chain {M}onte {C}arlo.
\newblock {\em Electronic Journal of Statistics}, 12:3397--3442.

\bibitem[Loader, 1999]{load:cliv:1999}
Loader, C.~R. (1999).
\newblock Bandwidth selection: classical or plug-in?
\newblock {\em Annals of Statistics}, 27:415--438.

\bibitem[Newey and West, 1987]{newe:west:1987}
Newey, W.~K. and West, K.~D. (1987).
\newblock A simple, positive semi-definite, heteroskedasticity and
  autocorrelation consistent covariance matrix.
\newblock {\em Econometrica}, 55:703--708.

\bibitem[Plummer et~al., 2006]{plum:best:cow:2006}
Plummer, M., Best, N., Cowles, K., and Vines, K. (2006).
\newblock {CODA}: convergence diagnosis and output analysis for {MCMC}.
\newblock {\em R news}, 6(1):7--11.

\bibitem[Politis, 2003]{politis2003adaptive}
Politis, D.~N. (2003).
\newblock Adaptive bandwidth choice.
\newblock {\em Journal of Nonparametric Statistics}, 15(4-5):517--533.

\bibitem[Politis, 2011]{politis:dimitris:2011}
Politis, D.~N. (2011).
\newblock Higher-order accurate, positive semidefinite estimation of
  large-sample covariance and spectral density matrices.
\newblock {\em Econometric Theory}, 27(4):703--744.

\bibitem[Politis and Romano, 1995]{poli:roma:1995}
Politis, D.~N. and Romano, J.~P. (1995).
\newblock Bias-corrected nonparametric spectral estimation.
\newblock {\em Journal of Time Series Analysis}, 16(1):67--103.

\bibitem[Politis and Romano, 1996]{poli:roma:1996}
Politis, D.~N. and Romano, J.~P. (1996).
\newblock On flat-top kernel spectral density estimators for homogeneous random
  fields.
\newblock {\em Journal of Statistical Planning and Inference}, 51(1):41--53.

\bibitem[Politis and Romano, 1999]{poli:roma:1999}
Politis, D.~N. and Romano, J.~P. (1999).
\newblock Multivariate density estimation with general flat-top kernels of
  infinite order.
\newblock {\em Journal of Multivariate Analysis}, 68(1):1--25.

\bibitem[Sheather and Jones, 1991]{shea:simo:jone:mich:1991}
Sheather, S.~J. and Jones, M.~C. (1991).
\newblock A reliable data-based bandwidth selection method for kernel density
  estimation.
\newblock {\em Journal of the Royal Statistical Society. Series B
  (Methodological)}, pages 683--690.

\bibitem[Silverman, 1999]{silv:1999}
Silverman, B.~W. (1999).
\newblock {\em Density Estimation for Statistics and Data Analysis}.
\newblock Chapman \& Hall Ltd.

\bibitem[Song and Schmeiser, 1995]{song:schm:1995}
Song, W.~T. and Schmeiser, B.~W. (1995).
\newblock Optimal mean-squared-error batch sizes.
\newblock {\em Management Science}, 41:110--123.

\bibitem[Taylor, 2018]{taylor:2018}
Taylor (2018).
\newblock Sum of autocovariances for {AR}(p) model.
\newblock Cross Validated.
\newblock URL:https://stats.stackexchange.com/q/372006 (version: 2018-10-17).

\bibitem[Thompson, 2010]{thom:2010}
Thompson, M.~B. (2010).
\newblock A comparison of methods for computing autocorrelation time.
\newblock {\em arXiv preprint arXiv:1011.0175}.

\bibitem[Tj{\o}stheim, 1990]{tjostheim1990non}
Tj{\o}stheim, D. (1990).
\newblock Non-linear time series and {M}arkov chains.
\newblock {\em Advances in Applied Probability}, pages 587--611.

\bibitem[Vats and Flegal, 2018]{vats:fleg:2018}
Vats, D. and Flegal, J.~M. (2018).
\newblock Lugsail lag windows and their application to {MCMC}.
\newblock {\em arXiv preprint arXiv:1809.04541}.

\bibitem[{Vats} et~al., 2018]{vats:fleg:jone:2015spec}
{Vats}, D., {Flegal}, J.~M., and {Jones}, G.~L. (2018).
\newblock Strong consistency of multivariate spectral variance estimators in
  {M}arkov chain {M}onte {C}arlo.
\newblock {\em Bernoulli}, 24:1860--1909.

\bibitem[{Vats} et~al., 2019]{vats:fleg:jone:2015output}
{Vats}, D., {Flegal}, J.~M., and {Jones}, G.~L. (2019).
\newblock Multivariate output analysis for {M}arkov chain {M}onte {C}arlo.
\newblock {\em Biometrika}, 106:321--337.

\bibitem[Woodroofe, 1970]{Wood:Mich:1970}
Woodroofe, M. (1970).
\newblock On choosing a delta-sequence.
\newblock {\em The Annals of Mathematical Statistics}, 41(5):1665--1671.

\end{thebibliography}
\end{document}